\RequirePackage{fix-cm}
\documentclass{article}
\usepackage{graphicx}
\usepackage{mathptmx}      
\usepackage{amsfonts}
\usepackage{amsmath}
\usepackage{amsthm}
\usepackage{amscd}
\usepackage{lmodern}
\usepackage{authblk}
\usepackage{enumitem}
\usepackage[latin2]{inputenc}
\usepackage{t1enc}
\usepackage[mathscr]{eucal}
\usepackage{indentfirst}
\usepackage{mathptmx}
\usepackage{graphicx}
\usepackage{graphics}
\usepackage{pict2e}
\usepackage{amssymb}
\usepackage{epic}
\numberwithin{equation}{section}
\usepackage[margin=2.9cm]{geometry}
\usepackage{epstopdf} 
\usepackage{bibcheck}
\usepackage{caption}
\usepackage{mdframed}
\usepackage{listings}
\usepackage{pifont}
\makeatother
\usepackage{subcaption}
\newcommand{\F}{\mathcal{F}}

\newcommand{\PWidth}{\tx{PWidth}}
\newcommand{\R}{\mathbb{R}}
\newcommand{\DSB}{\tx{DSB}}
\newcommand{\f}{\tilde{f}}
\newcommand{\sa}{\sigma_{\alpha}}

\newcommand{\m}{m}

\newcommand{\tx}{\textnormal}
\newcommand{\B}{\bar{B}}

\newcommand{\n}[1]{\|#1\|}
\newcommand{\argmin}{\tx{argmin}}
\newcommand{\argmax}{\tx{argmax}}
\newcommand{\x}{\bar{x}}

\newcommand{\dist}{\tx{dist}}
\newcommand{\A}{\mathcal{A}}
\newcommand{\s}[2]{\langle #1, #2 \rangle}
\newcommand{\OM}{\Omega}
\newcommand{\K}{K}
\newcommand{\conv}{\tx{conv}}
\newcommand{\sm}{\setminus}
\newcommand{\SB}{\tx{SB}}
\newcommand{\bn}{\bar{\nu}}
\newcommand{\ro}{g_{\OM}}
\newcommand{\rj}{g_J}

\newcommand{\tip}{P}

\newcommand{\bA}{\bar{A}}
\newcommand{\bgo}{\bar{g}_{\OM}}
\newcommand{\bg}{\bar{g}}
\newcommand{\bj}{\bar{g}_J}
\newcommand{\bh}{\bar{h}}
\newcommand{\vg}{v_g}
\newcommand{\bvg}{\bar{v}_g}
\newcommand{\T}{\bar{\tau}}
\newcommand{\sop}{\tx{SOR}}
\newcommand{\bsop}{\overline{\sop}}
\newcommand{\hn}{\hat{\nu}}
\newcommand{\tom}{\bar{T}(\OM)}
\newcommand{\bara}{\bar{\alpha}}

\newtheorem{proposition}{Proposition}[section]
\newtheorem{lemma}{Lemma}[section]
\newtheorem{corollary}{Corollary}[section]
\theoremstyle{definition}
\newtheorem{remark}{Remark}[section]
\theoremstyle{definition}
\newtheorem{definition}{Definition}[section]
\begin{document}
\title{A unifying framework for the analysis of  projection-free first-order methods under a sufficient slope condition}

	\author{Francesco Rinaldi\thanks{Dipartimento di Matematica "Tullio Levi-Civita", Universit\`a di Padova, Italy (rinaldi@math.unipd.it)  
		}	\quad \quad \quad
	Damiano Zeffiro\thanks{Dipartimento di Matematica "Tullio Levi-Civita", Universit\`a di Padova, Italy (zeffiro@math.unipd.it)
	}
	}
	\date{}
\maketitle
	\begin{abstract}
		The analysis of projection-free first order methods is often complicated by the presence of different kinds of "good" and "bad" steps. In this article, we propose a unifying framework for projection-free methods, aiming to simplify the converge analysis by getting rid of such a distinction between steps.
		The main tool employed in our framework is the Short Step Chain (SSC) procedure, which skips gradient computations in consecutive short steps until proper stopping conditions are satisfied. This technique allows us to give a unified analysis and converge rates in the general smooth non convex setting, as well as convergence rates under a Kurdyka-\L ojasiewicz (KL) property, a setting that, to our knowledge, has not been analyzed before for the projection-free methods under study. In this context, we prove local convergence rates comparable to those of projected gradient methods under the same conditions.
	 
		Our analysis relies on a sufficient slope condition, ensuring that the directions selected by the methods have the steepest slope possible up to a constant among feasible directions. This condition is satisfied, among others, by several Frank-Wolfe (FW) variants on polytopes, and by some projection-free methods on convex sets with smooth boundary.  
	\end{abstract}
	
	\section{Introduction}
	Projection-free first-order methods aim to provide a valid alternative to projected gradient approaches for the constrained optimization of a smooth objective $f: \R^n \rightarrow \R$, in settings where projecting on the feasible set may be unpractical. Popular methods of this kind are the FW algorithm \cite{frank1956algorithm} and its variants (see, e.g., \cite{freund2017extended}, \cite{lacoste2015global} and references therein), which have many applications in sparse and structured optimization (see, e.g., \cite{freund2017extended}, \cite{jaggi2013revisiting}, \cite{joulin2014efficient}, 
	\cite{osokin2016minding}  and references therein).
	
	In this paper, we present a specific algorithmic framework that can easily embed 
	a number of projection-free (deterministic) methods proposed in the literature. This enables us to give a unified theoretical analysis for the methods at hand and new convergence results in the non convex case. More specifically, the framework we propose aims to overcome an annoying issue affecting the analysis of projection-free methods, that is the presence of "bad iterations", i.e., iterations where we cannot show good progress. This happens when we are forced to take a short step along the search direction to guarantee feasibility of the iterate.  The number of short steps typically needs to be upper bounded in the convergence analysis with "ad hoc" arguments (see, e.g., \cite{freund2017extended} and \cite{lacoste2015global}). The main idea behind our method is to chain several short steps by skipping gradient updates until proper stopping conditions are met. This has the additional benefit to rule out unpractical convergence rates due to large numbers of short steps in the analysis of the methods, like, for instance, in the PFW algorithm on the $N - 1$ dimensional simplex,  where the best known upper bound on the number of consecutive short steps is $3N!$ \cite{lacoste2015global}.	\\
	\indent The projection-free first order methods we consider satisfy a sufficient slope condition related to the slope along descent directions. Essentialy, for any direction $d_x$ selected by the method in $x$ and any direction $d$ feasible in $x$ we must have
	\begin{equation*}
		- \s{\nabla f(x)}{\frac{d_x}{\n{d_x}}} \geq - \tau \s{\nabla f(x)}{\frac{d}{\n{d}}} 
	\end{equation*}
	for a fixed $\tau > 0$ (independent from $x$ and $\nabla f$). Thus any direction selected by the method must be, up to a fixed constant, the steepest descent direction among feasible directions. The sufficient slope condition rules out the well known zig zagging behaviour of the classic FW method, whose directions can be almost orthogonal to the gradient on the interior of the feasible domain, leading to a slow convergence of $O(1/k)$ even for smooth strongly convex objectives \cite{canon1968tight, frank1956algorithm, lacoste2015global, wolfe1970convergence}. We show that this condition applies to the AFW, the PFW and the FW method with in face directions (FDFW) (see, e.g., \cite{freund2017extended} and \cite{guelat1986some}) on polytopes, to a retraction based approach on convex sets with smooth boundary as the one proposed in \cite{levy2019projection}, to the FDFW method on sublevel sets of smooth strongly convex functions, and to tailored approaches over product domains.  \\
	\indent We analyze convergence rates for non convex smooth objectives both in the general case and under the Kurdyka-\L ojasiewicz  property (\cite{attouch2010proximal}, \cite{bolte2007clarke}, \cite{bolte2010characterizations}; see also Section \ref{KLproperty} for a review of the essential definitions). This property holds for semialgebraic functions and more in general for functions definable in an $o$-minimal class (see, e.g., \cite{attouch2010proximal}, \cite{bolte2007clarke}). While there are many works on the convergence rates of proximal subgradient type methods under KL properties (see, e.g., \cite{attouch2010proximal}, \cite{attouch2013convergence},  \cite{bolte2017error}, \cite{wang2019global}, \cite{xu2013block}), we are not aware of previous analyses including projection-free methods for convex sets and smooth non convex objectives. In the convex setting, FW variants have been analyzed on polytopes \cite{kerdreux2018restarting} and uniformly convex sets (see, e.g., \cite{kerdreux2020projection}, \cite{xu2018frank}) under H\"{o}lderian error bound conditions, to interpolate between the well known $O(1/k)$ rate for general convex objectives and the faster rates for strongly convex ones proved in \cite{garber2015faster} and \cite{lacoste2015global}. These H\"{o}lderian error bound conditions are equivalent to certain KL properties for convex objectives, as proved for a more general class of error bounds in \cite{bolte2017error}. As for the smooth non convex case, in \cite{liu2016quadratic} a KL property leading to linear convergence of some methods using retractions is proved for a quadratic problem on the Stiefel manifold; in~\cite{balashov2019gradient} the convergence of a method using orthographic retractions is analyzed on proximally smooth surfaces under a KL property (see also Remark \ref{LPLrem}). \\
	\indent When dealing with the FW method and its variants, which represent a specific class of projection-free first-order methods, the SSC procedure has the additional advantage that it skips the computation of the solutions of linear programming subproblems. Different approaches to avoid computations of those solutions have been discussed in \cite{braun2018blended} , \cite{braun2017lazifying} for convex optimization over polytopes. The blended conditional gradient introduced in \cite{braun2018blended} uses projections on the convex hull of the current active set as an alternative to pairwise or away FW steps, but its analysis still allows for short steps. Other works in projection-free optimization (e.g. \cite{hazan2016variance}, \cite{lan2016conditional}) rely on the classic FW method to approximately compute the projection in accelerated projected gradient methods for convex objectives. To our knowledge this approach, called conditional gradient sliding, has not yet been combined effectively with faster FW variants, and it  does not currently lead to any improvement on the $O(1/\varepsilon)$ linear oracle complexity of the FW method. Our methods using the SSC procedure instead have a complexity of $O(\tx{ln}\frac{1}{\varepsilon})$  both in gradients and linear oracles for non convex objectives satisfying a KL property implied e.g. by the Luo Tseng error bound \cite{luo1993error} under some mild separability conditions for the stationary points \cite[Theorem 4.1]{li2018calculus}. This error bound is known to hold in a variety of convex and non convex settings (see references in \cite{li2018calculus}).  \\	
	\indent The structure of the paper is as follows. In Section 2, we define some notation and state some preliminary results from convex analysis. In Section 3, we introduce a sufficient slope property for first-order projection free methods. We define the SSC procedure in Section 4, and prove it can be applied to generate some subgradient descent-like  sequences (see, e.g., \cite{attouch2013convergence}, \cite{bolte2017error}, \cite{luo1993error}) adapted to our particular setting. In Section 5, we state the main convergence results.  In Section~6, we show examples of methods fitting our framework. Finally, some proofs, technical lemmas and background material can be found in the appendix. 
	
	\section{Notation and preliminaries}\label{s:prel} 
	We consider the following constrained optimization problem:
	\begin{equation}\label{eq:mainpb}
		\tx{min} \left\{f(x) \ | \ x \in \Omega \right\} \, .
	\end{equation}
	In the rest of the article $\Omega$ is a compact and convex set, unless specified otherwise, and $f \in C^1(\OM)$ with $L$-Lipschitz gradient:
	$$ \n{\nabla f(x) - \nabla f(y)} \leq L\n{x - y} \tx{ for all }x, y \in \OM \, .$$
	We define $D$ as the diameter of $\OM$, $\hat{c} = c/\|c\|$ for $c \in \mathbb{R}^n / \{ 0 \} $ and $\hat{c} = 0$ for $c = 0$. For sequences we write $\{x_k\}$ instead of $\{x_k\}_{k \in I}$  when $I$ is clear from the context, with $I = \mathbb{N}_0$ by default and $[1:m] = \{1, ...,m\}$. Given a function $\tilde{f}$ with values in $\R$ we use the notation 
	\begin{equation*}
		[\eta_1 < \tilde{f}< \eta_2] = \{x \in \R^n \ | \ \eta_1 < \tilde{f}(x) < \eta_2 \}
	\end{equation*}
	for $-\infty \leq \eta_1 < \eta_2 \leq +\infty$ and we use the same notation for non strict inequalities. 
		
For subsets $C, D$ of $\R^n$ we define $\dist(C, D)$ as
	\begin{equation*}
		\dist(C, D) = \inf \{\n{y-z} \ | \ z \in C, \ y \in D \} \, ,
	\end{equation*}
	$B_R(C)$ as the neighborhood $\{x \in \R^{n} \ | \ \dist(C, x) < R \}$  of $C$ of radius $R$ and in particular $B_R(x)$ as the open euclidean ball of radius $R$ and center $x$. When $C$ is closed and convex we define as $\pi(C, \cdot)$ the projection on $C$. If $C$ is a cone then we denote with $C^*$ its polar. 
	
	We now state some elementary properties related to the tangent and the normal cones, where for $\x\in \Omega$ we denote with $T_{\Omega}(\bar{x})$ and $N_{\OM}(\bar{x})$ the tangent and the normal cone to $\Omega$ in $\x$ respectively. The next proposition (from \cite{rockafellar2009variational}, Theorem 6.9) characterizes these cones for closed convex subsets of $\R^n$. 
	\begin{proposition} \label{normalC}
		Let $\Omega$ be a closed convex set. For every point $\bar{x}\in \Omega$ we have
		\begin{align*}
			&T_{\Omega}(\bar{x}) = \textnormal{cl}\{w \ | \ \exists \lambda>0 \textnormal{ with } \bar{x}+\lambda w \in \Omega \} \, , \\
			\textnormal{int}&(T_{\Omega}(\bar{x})) = \{w \ | \ \exists \lambda > 0 \textnormal{ with } \bar{x}+\lambda w \in \textnormal{ int}(\Omega)\} \, , \\ 
			& N_{\Omega}(\bar{x}) = T_{\Omega}(\bar{x})^*=\{v\in \R^n\ | \ (v,y - \x)\leq 0 \ \forall\ y \in  \OM \} \, .
		\end{align*}	
	\end{proposition}
	We have the following formula connecting the maximal "slope" of a linear function along feasible directions to the tangent and the normal cone: 
	\begin{proposition} \label{NWessential}
		If $\Omega$ is a closed convex subset of $\R^n$, $\bar{x} \in \Omega$ then for every $g \in \R^n$
		$$	\max \left \{0, \sup_{h\in\Omega\sm \{\bar{x}\}} \left (g, \frac{h-\bar{x}}{\|h - \bar{x}\|}\right )\right \} = \textnormal{dist}(N_{\Omega}(\bar{x}), g) = \|\pi(T_{\Omega}(\bar{x}), g)\| \, .$$	
	\end{proposition}
	This property is a consequence of the Moreau-Yosida decomposition \cite{rockafellar2009variational} and we refer the reader to Section \ref{PAp} in the Appendix for a detailed proof. In the rest of the article to simplify notations we often use $\pi_{\x}(g)$ as a shorthand for $\|\pi(T_{\Omega}(\bar{x}), g)\|$, so that by Proposition \ref{NWessential} first order stationarity conditions in $\x$ for the gradient $-g$ become equivalent to $\pi_{\x}(g) = 0$.
	\section{Sufficient slope condition} \label{SBLB}
	Let $\mathcal{A}$ be a first-order optimization method defined for smooth functions on a closed subset $\OM$ of $\R^n$. We assume that given first-order information $(x_k, \nabla f(x_k))$ the method always selects $x_{k + 1}$ along a feasible descent direction, so that for $(x, g) \in \OM \times \R^n$ we can define 
	\begin{equation*}
		\A(x,g) \subset T_{\OM}(x) \cap \{y \in \R^n \ | \ \s{g}{y} > 0 \} \cup \{0\} \,
	\end{equation*}
	as the possible descent directions selected by $\A$ when $x= x_k$, $g = -\nabla f(x_k)$ for some $k$ (see Algorithm~\ref{tab:1}). When $x$ is first-order stationary, we set $\A(x, g) = \{ 0 \}$, otherwise we always assume $0 \notin \A(x, g)  $.   
	\begin{center}
		\begin{table}[h]
			\caption{First-order method}
			\label{tab:1}
			\begin{center}
				\begin{tabular}{l}
					\hline\noalign{\smallskip} 
					\textbf{Initialization.} $x_0 \in \OM$, $k := 0$. \\		
					1. If $x_k$ is stationary, then STOP  \\
					2. select a descent direction  $d_k \in \A(x_k, - \nabla f(x_k))$ \\
					3. set $x_{k+1} = x_k + \alpha_k d_k$ for some stepsize $\alpha_k \in [0, \alpha_k^{\max}]$ \\
					4. set $k := k+1$, go to Step 1. \\
					\noalign{\smallskip} \hline
				\end{tabular}
			\end{center}	
		\end{table}
	\end{center}

	We want to formulate a "sufficient slope" condition for the descent directions selected by $\mathcal{A}$ with respect to the steepest slope among feasible directions.
	In order to do that, we define the directional slope lower bound as
	\begin{equation*}
		\tx{DSB}_{\A}(\OM, x, g) = \inf_{d \in \A(x,g)} \frac{\s{g}{d}}{\pi_x(g) \n{d}} 
	\end{equation*}
	if $0 \notin \A(x,g)$.  Otherwise by our assumptions $x$ is stationary for $-g$, $\pi_x(g) = 0$ and we set $\tx{DSB}_{\mathcal{A}}(\OM, x, g) = 1$. With these definitions for any $x \in \OM$, $g \in \R^n$
	\begin{equation} \label{DSBineq}
		\tx{DSB}_{\A}(\OM, x, g) \leq 1
	\end{equation}
	since if $\pi_x(g) \neq 0$ then for any $d$ feasible descent direction $\frac{\s{g}{d}}{\n{d}} \leq \pi_x(g)$  by Proposition \ref{NWessential}, and if $\pi_x(g) = 0$ by definition $\tx{DSB}_{\A}(\OM, x, g) = 1$. Given a subset $P$ of $\OM$ we can finally define the slope lower bound
	\begin{equation*}
		\tx{SB}_{\A}(\OM, P) = \inf_{\substack{g \in \R^n \\x \in P}} \tx{DSB}_{\A}(\OM, x, g) = \inf_{\substack{g: \pi_x(g) \neq 0 \\x \in P}} \tx{DSB}_{\A}(\OM, x, g)  \, ,
	\end{equation*}
	where in the second equality we used that \eqref{DSBineq} holds with equality if $\pi_x(g) = 0$.
	For simplicity if $P = \OM$ we write $\tx{SB}_{\A}(\OM)$ instead of $\tx{SB}_{\A}(\OM, \OM)$.
		
	In Section \ref{s:examples}, we show a few examples of methods and sets satisfying a sufficient slope condition, i.e., cases where the  slope lower bound is strictly grater than zero.
	\section{First order projection free methods with SSC procedure} 	
	In this section, we briefly recall the subgradient descent sequence conditions used in the analysis of a variety of descent methods (see, e.g., \cite{attouch2013convergence}, \cite{bolte2017error}, \cite{luo1993error}), with a few simple adaptations to our smooth constrained projection-free setting. We then introduce the SSC procedure, and show how it allows us to obtain our adapted subgradient descent sequences. 	
	\subsection{Subgradient descent sequences for projection-free methods}
	
	Let $\tilde{f}: \R^n \rightarrow \R \cup \{+\infty\}$ be a proper lower semicontinuous function and let $\partial \tilde{f}$ be its limiting subdifferential. We refer the reader to Section \ref{KLproperty} of the appendix for a brief reivew of the Fr\'echet subdifferential and the related KL property as defined in \cite{attouch2010proximal}, \cite{attouch2013convergence}. In this work we are only interested in the case $\tilde{f} = f + i_{\OM}$ with $i_{\OM}$ the indicator function for $\OM$. We anyway notice that the abstract convergence result we give in Section \ref{s:CP} holds for general lower semicontinuous functions.

	Subgradient descent sequences as defined in \cite{bolte2017error} satisfy the following properties for some given $a, b> 0$:
	\begin{subequations}
		\begin{align}
			\tilde{f}(x_{k + 1}) + a \n{x_{k +1} - x_{k} }^2 & \leq \tilde{f}(x_{k}) \tag{H1},  \label{eq:H1}\\	
			\exists \ w_{k+1} \in \partial \tilde{f}(x_{k+1}) \tx{ such that } \n{w_{k+1}} & \leq b\n{x_{k+1}-x_{k}} \tag{H2} \label{eq:H2}	 \, .
		\end{align}
	\end{subequations}
	Since $\partial \f (x_k)$ is closed, condition \eqref{eq:H2} is true iff $\partial \f (x_{k+1})\neq \emptyset$ and
	\begin{equation} \label{eq:H2bis}
		\dist(\partial \f (x_{k+1}), 0 ) \leq b\n{x_{k+1}-x_{k}} \, .
	\end{equation}
	In our smooth constrained setting,  that is when $\tilde{f} = f + i_{\OM}$, by \eqref{eqdpi} we can rewrite \eqref{eq:H2bis} as
	\begin{equation} \label{eq:H2bissmooth}
		\n{\pi(T_{\OM}(x_{k + 1}), -\nabla f(x_{k + 1}))} \leq b\n{x_{k+1}-x_{k}} \, .
	\end{equation}
	However, for projection-free methods equation \eqref{eq:H2bis} and in particular \eqref{eq:H2bissmooth} do not hold in general, even in linear convergence settings. 
	
	\begin{figure}[h]
		\centering
		\begin{subfigure}[t]{0.45\textwidth}
			\includegraphics[width=\textwidth]{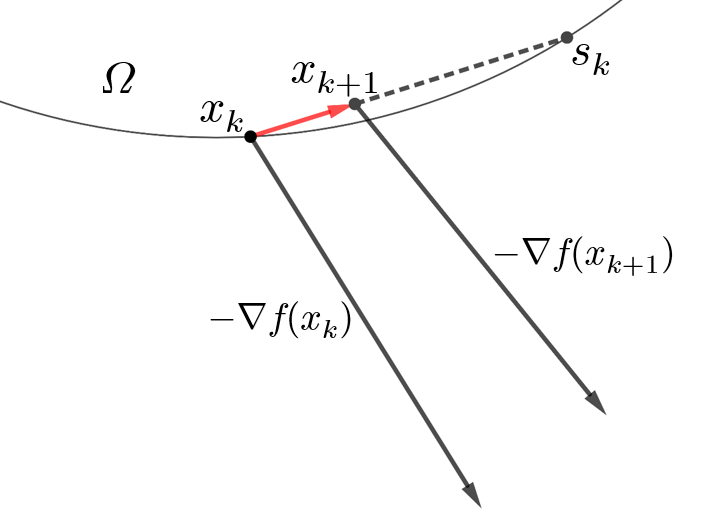} 
			\caption{FW method \eqref{eq:FWstep} does a full step from the \\ boundary of a sphere. $s_k \in \argmin_{s \in \OM} \s{\nabla f(x)}{s}. $}
			\label{fig:subim1}
		\end{subfigure}
\hfill
		\begin{subfigure}[t]{0.45\textwidth}
			\includegraphics[width=\textwidth]{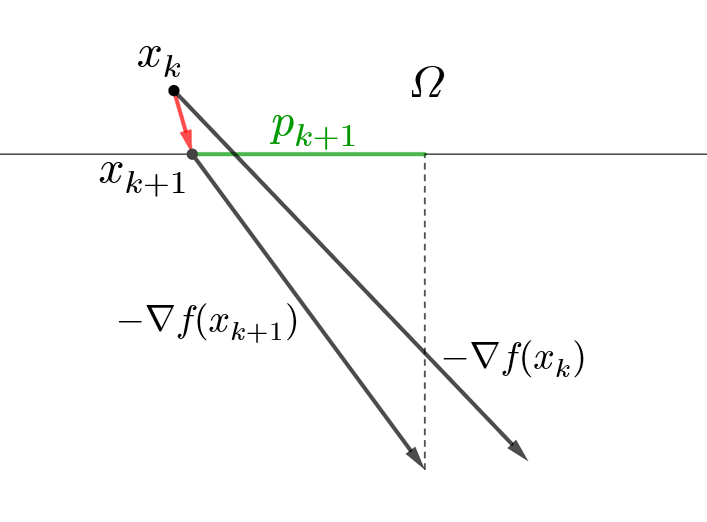}
			\caption{AFW method \eqref{AFWdir} does a short step \\ near the boundary of a polytope. }
			\label{fig:subim2}
		\end{subfigure}
		\caption{}
		
	\end{figure}

	Intuitively, for these methods the first obstacle for condition \eqref{eq:H2bissmooth} or variants is that if $x_{k} \in \partial \OM$ and  $x_{k + 1} \in \tx{int}(\OM)$ with $\n{\pi(T_{\OM}(x_{k}) -\nabla f(x_{k}))} \ll  \n{\nabla f(x_{k})} $ we have, for typical choices of the stepsize (see figure \ref{fig:subim1}):
	\begin{equation*}
		\n{x_{k +  1} - x_k} \thickapprox \n{\pi(T_{\OM}(x_{k}), -\nabla f(x_{k}))} \ll  \n{\nabla f(x_{k})} \thickapprox \n{\nabla f(x_{k + 1})} = \n{\pi(T_{\OM}(x_{k + 1}), -\nabla f(x_{k + 1}))} \, .
	\end{equation*}
	This leads us to introduce the weaker condition
	\begin{equation} \label{eq:gd'2}
		\f (x_k) - \f (x_{k+1}) \geq \frac{a}{b^2} \dist(\partial \f (\tilde{x}_{k}), 0 )^2 \tag{H'2} \, ,
	\end{equation}
	for some auxiliary point $\tilde{x}_k$ s.t. $\partial \f (\tilde{x}_k) \neq \emptyset$ and
	\begin{equation} \label{eq:gd''2}
		\f (x_{k + 1}) \leq \f (\tilde{x}_k) \leq \f (x_k) - a \n{\tilde{x}_k - x_k}^2 \, .
	\end{equation}
	In particular in the setting of figure \ref{fig:subim1} a typical choice is $\tilde{x}_k = x_k$, as we shall see more formally in the next section. In the rest of this article \eqref{eq:gd'2} always comes with \eqref{eq:gd''2} for $\tilde{x}_k$. 
	
	As anticipated, condition \eqref{eq:gd'2} is weaker than \eqref{eq:H2} if one is already assuming \eqref{eq:H1}. Indeed if \eqref{eq:H1} and \eqref{eq:H2} hold then
	\begin{equation*}
		\f (x_k) - \f (x_{k+1}) \geq a \n{x_k - x_{k+1}}^2  \geq \frac{a}{b^2}	\dist(\partial \f (x_{k+1}), 0 )^2  \, ,
	\end{equation*}
	where the second inequality directly follows from \eqref{eq:H2} in the form \eqref{eq:H2bis}. One can then satisfy \eqref{eq:gd'2} and \eqref{eq:gd''2} by taking $\tilde{x}_k = x_{k+1}$. 
	
	In our setting, that is when $\tilde{f} = f + i_{\OM}$, we use the smooth version of \eqref{eq:gd'2}:
	\begin{equation} \label{d'2}
		f(x_k) - f(x_{k+1}) \geq \frac{a}{b^2} \n{\pi(T_{\OM}(\tilde{x}_k), -\nabla f(\tilde{x}_k))}^2 \tag{H'2a} \, ,
	\end{equation}
	with analogous conditions on $\tilde{x}_k$.
	\subsection{The SSC procedure} \label{s:SSC} 
	
	The second (and arguably most important) issue for applying conditions such as \eqref{eq:H2bissmooth} to projection-free methods  is the presence of short steps. Given a projection-free first-order method $\mathcal{A}$, we say that the step $k$ is short if the stepsize $\alpha_k$ is maximal, that is $ \alpha_k= \alpha_k^{\max}$. Roughly speaking, when $x_k$ is close enough to $\partial \OM$ and $d_k$ points toward $\partial \OM$ we have (see figure \ref{fig:subim2} with $p_{k + 1}=\pi(T_{\OM}(x_{k + 1}), -\nabla f(x_{k + 1}))$)
	\begin{equation*}
		\n{x_{k +  1} - x_k} = \alpha_k \n{d_k} = \alpha_k^{\max} \n{d_k} \ll  \n{\pi(T_{\OM}(x_{k + 1}), -\nabla f(x_{k + 1}))} \, .
	\end{equation*}    
	An analogous issue affects also weaker variants of \eqref{eq:H2bissmooth}, including \eqref{d'2}. In order to address this issue, we apply the SSC procedure to first order methods (Algorithm~\ref{tab:3}). This procedure (Algorithm~\ref{tab:4}) chains consecutive short steps, skipping updates for the gradient and possibly for related information, like linear minimizers, until proper stopping conditions are met. After the SSC returns a candidate point $x_{k+1}^{tr}$, the method may still choose another point in Step 4, selecting a $x_{k + 1}$ with a smaller objective value than $x_{k + 1}^{tr}$ and satisfying \eqref{eq:H1}:
	\begin{equation} \label{eq:step4}
		f(x_{k + 1}) \leq \min(f(x_{k + 1}^{tr}), f(x_k) - \frac{L}{2} \n{x_{k + 1} - x_k}^2) \, .
	\end{equation}	
	The method can always take $x_{k + 1}^{tr}$, since, as we prove later in this section, the following inequality holds
	$$f(x_{k+1}^{tr}) \leq f(x_k) - \frac{L}{2} \n{x_{k + 1}^{tr} - x_k}^2 \, . $$
	
	\begin{center}
		\begin{table}[h]
			\caption{First-order method with SSC}
			\label{tab:3}
			\begin{center}
				\begin{tabular}{l}
					\hline\noalign{\smallskip} 
					\textbf{Initialization.} $x_0 \in \OM$, $k = 0$.      \\
					1. \textbf{while} $x_k$ is not stationary: \\
					2. \quad $g = -\nabla f(x_k)$ \\
					3. \quad $x^{tr}_{k+1} = \textnormal{SSC}(x_k, g)$ \\
					4. \quad select $x_{k + 1} \in \OM$ satisfying \eqref{eq:step4} \\				
					5. \quad $k = k+1$. \\
					\noalign{\smallskip} \hline
				\end{tabular}
			\end{center}	
		\end{table}
	\end{center}  
	
	\begin{center}
		\begin{table}[h]
			\caption{SSC$(\x, g)$ }
			\label{tab:4}
			\begin{center}
				\begin{tabular}{l}
					\hline\noalign{\smallskip} 
					\textbf{Initialization.} $y_0 = \x$, $j=0$.  \\	
					\textbf{Phase I} \\
					
					1. \quad select $ d_j \in \A(y_j, g)$, $\alpha^{(j)}_{\max} \in \alpha_{\max}(y_j, d_j)$ \\
					2. \quad \textbf{if } $d_j = 0$ \textbf{then:} \\
					3. \quad \quad \textbf{return } $y_j$ \\
					\textbf{Phase II}		\\
					4. \quad compute $\beta_j$ with \eqref{betaj} \\					 		 		
					5. \quad let $\alpha_j = \min(\alpha^{(j)}_{\max}, \beta_j)$ \\
					6. \quad \textbf{if} $\alpha_j = 0$ \ \textbf{then}: \\
					7. \quad \quad  \textbf{return } $y_j$ \\
					\textbf{Phase III} \\
					8. \quad $y_{j+1} = y_j + \alpha_j d_j$ \\
					9. \ \ \textbf{if} $\alpha_j =  \beta_j$  \textbf{then}: \\
					10. \ \ \quad  \textbf{return } $y_{j+1}$ \\
					11. \ \ $j = j+1$, go to Step 1. \\
					\noalign{\smallskip} \hline
				\end{tabular}
			\end{center}	
		\end{table}
	\end{center}
Given that the gradient $-g$ is constant during the SSC, this procedure can be viewed as an application of $\A$ for the minimization of the linearized objective $f_g(z) = \s{ - g}{z - \bar{x}} + f(\bar{x})$ with peculiar stepsizes and stopping criterion. Fundamentally, these must ensure that the linearized objective is still a sufficiently accurate approximation of $f$ on the generated sequence, and that the directions selected are descent directions not only for $f_g$ but also for the true objective $f$. 

	We can split the SSC into three phases. First, the SSC computes a feasible descent direction $d_j~\in~\mathcal{A}(y_j, g)$ if it exists, otherwise $d_j = 0$ and the algorithm returns $y_j$. Notice that $d_j = 0$ iff $y_j$ is stationary w.r.t. the approximated gradient $-g$. Second, the stepsize $\alpha_j$ is computed by taking the minimum between the maximal stepsize $\alpha^{(j)}_{\max}$ and an auxiliary stepsize $\beta_j$. The maximal stepsize $\alpha^{(j)}_{\max}$ may depend on boundary conditions but also on additional information available like the active set in the AFW and in the PFW. In the theoretical analysis, we only need to assume termination of the procedure and  $\alpha^{(j)}_{\max} > 0$. In particular, $\alpha_j = 0$ in Phase II iff $\beta_j = 0$, or equivalently if $d_j$ is not feasible from $y_j$ for the auxiliary set $\OM_j$ (see \eqref{eq:omegaaux}). As we shall see later in this section, when $d_0 \neq 0$ we always have $\beta_0 > 0$, and therefore Phase II can never return $y_0$ (which could cause an infinite loop in Algorithm \ref{tab:3}). 
	In the third and final phase, the new point of the SSC $y_{j+1} = y_j + \alpha_j d_j$ is computed. Then the procedure terminates if $\alpha_j = \beta_j$, otherwise it continues without updating the gradient. 
	
	We now define the auxiliary step size $\beta_j$, which allows us to retrieve the descent sequence conditions \eqref{eq:H1} and \eqref{d'2}.  Let 
	\begin{equation} \label{eq:omegaaux}
		\OM_j =   \B_{\n{g}/2L}(\bar{x} + \frac{g}{2L}) \cap \B_{\s{g}{\hat{d}_j}/L}(\bar{x}) \, .
	\end{equation} 
	We define 
	\begin{equation} \label{betaj}
		\beta_j = 
		\begin{cases}
			0 &\tx{ if } y_j \notin \OM_j \, , \\
			\beta_{\tx{max}}(\OM_j, y_j, d_j) & \tx{ if } y_j \in \OM_j \, .
		\end{cases}
	\end{equation}
	where $\beta_{\max}(\OM_j, y_j, d_j) = \max \{\beta \in \R_{\geq 0} \ | \ y_j + \beta d_j \in \OM_j \}$ is the maximal feasible stepsize in the direction $d_j$ starting from $y_j$ with respect to $\OM_j$. Since $\OM_j$ is the intersection of two balls there is a simple closed form expression for $\beta_j$. Given that $y_0 = \bar{x}$ it is immediate to check that if $d_0 \neq 0$
	\begin{equation*}
		\beta_0 = \frac{\s{g}{\hat{d}_0}}{L\n{d_0}}\, ,
	\end{equation*}
	which corresponds to the constant stepsize $1/L$ for gradient descent (when $d_0 = g$), and where $\beta_0 > 0$ since $d_0 \neq 0$ is by assumption a descent direction for $-g$.
	
	The second ball appearing in the definition of $\OM_j$ does not depend on $j$ and can be rewritten as
	\begin{equation} \label{barBeq}
		\B_{\n{g}/2L}\left(\bar{x} + \frac{g}{2L}\right) = \{ z \in \R^n \ | \ L\n{z-\bar{x}}^2 - \s{z - \bar{x}}{g} \leq 0 \} \, .
	\end{equation}
	Then for any $z \in 	\B_{\n{g}/2L}\left(\bar{x} + \frac{g}{2L}\right) $ we have
	\begin{equation} \label{stdbarB}			
		f(z) \leq f(\bar{x}) - \s{g}{z-\bar{x}} + \frac{L}{2}\n{z - \bar{x}}^2 \leq f(\bar{x}) - \frac{L}{2}\n{\bar{x} - z}^2	\, ,	
	\end{equation}
	where the first inequality is the standard descent lemma and the second follows from \eqref{barBeq}.  Let now $\bar{B} = \B_{\n{g}/2L}(x_k + \frac{g}{2L}) $.  Since $y_0 = x_k \in \bar{B}$, and if $y_j \in \bar{B} $ then $y_{j+1} \in \bar{B}$ by definition of $\beta_j \geq \alpha_j$, by induction we obtain $y_j \in \bar{B}$ for every $j \in \{0, ..., T\}$.  Therefore thanks to \eqref{stdbarB} for every $j \in \{0, ..., T\}$ we have the sufficient decrease condition
	\begin{equation} \label{eq:dcondy}
	f(y_j)  \leq f(x_k) - \frac{L}{2}\n{x_k - y_j}^2 \, .
	\end{equation}
	We now show how the condition $y_{j+1} \in \OM_j \subset \B_{\s{g}{\hat{d}_j}/L}(\bar{x})$ guarantees that the true objective $f$ is monotone decreasing during the SSC. 
	\begin{lemma} \label{l:decreasing}
		Assume $y_j \in \B_{\s{g}{\hat{d}_j}/L}(\bar{x})$. Then for every $\beta \in [0, \beta_j]$ we have 
		\begin{equation*}
			\frac{d}{d\beta} f(y_j + \beta d_j) \leq 0 \, ,
		\end{equation*}	
		and thus in particular $f(y_j + \beta_j d_j) \leq f(y_j)$.
	\end{lemma}
	\begin{proof}
		We have 
		\begin{equation*}
			\begin{aligned}
				& \frac{d}{d\beta} f(y_j + \beta d_j) = \n{d_j} \s{\nabla f(y_j + \beta d_j)}{\hat{d}_j}  \\
				= & \n{d_j} \s{ (\nabla f(y_j + \beta d_j) + g) - g}{\hat{d}_j} = \n{d_j} (\s{\nabla f(y_j + \beta d_j) + r}{\hat{d}_j} - \s{g}{\hat{d}_j}) \\ 
				\leq & \n{d_j} (L\n{\bar{x}-y_j - \beta d_j} - \s{g}{\hat{d}_j}) \leq 0 \, ,
			\end{aligned}		 
		\end{equation*}
		where we used $g = -\nabla f(\bar{x})$ and the Lipschitzianity of $\nabla f$ in the first inequality and $$y_j + \beta d_j \in \B_{\s{g}{\hat{d}_j}/L}(\bar{x})$$ in the second. 
	\end{proof}
	
	We can now prove the descent sequence conditions \eqref{eq:H1} and \eqref{d'2} for a sequence $\{x_k\}$ generated by Algorithm \ref{tab:3}. 
	\begin{proposition} \label{keyprop}
		Let $\tau = \text{SB}_{\mathcal{A}}(\OM) > 0$. Assume the SSC is given in input $(x_k, g)$ and let $x_{k +  1}^{tr}$ be the output. Then
		\begin{align}
			f(x_k) - f(x_{k +  1}^{tr}) & \geq \frac{L}{2}\n{x_k - x_{k +  1}^{tr}}^2 \, , \label{dq} \\
			\n{x_k - x_{k +  1}^{tr}} & \geq  \K \n{\pi(T_{\OM}(\tilde{x}_k), -\nabla f(\tilde{x}_k))} \, . \label{dq2}
		\end{align}	      
		for some $\tilde{x}_k$ such that $f(x_{k +  1}^{tr}) \leq f(\tilde{x}_k) \leq f(x_k) - \frac{L}{2} \n{x_k - \tilde{x}_k}^2$ and for $\K =\tau/(L(1+\tau))$. Furthermore, conditions \eqref{eq:H1} and \eqref{d'2} are satisfied with $b = \K$ and $a = L/2$. 
	\end{proposition}
	\begin{proof}
		Let $B_j = \B_{\s{g}{\hat{d}_j}/L}(x_k)$ and let $T$ be such that $x_{k +  1}^{tr} = y_T$.  \\ 
		Inequality \eqref{eq:dcondy} applied with $j = T$ gives \eqref{dq}. Moreover, by taking $\tilde{x}_k = y_{\tilde{T}}$ for some $\tilde{T} \in \{0, ..., T\}$ the conditions 
		\begin{equation} \label{xktildecond}
			f(x_{k+1}^{tr}) \leq f(\tilde{x}_k) \leq f(x_k) - \frac{L}{2} \n{x_k - \tilde{x}_k}^2 
		\end{equation}
		are satisfied by Lemma \ref{l:decreasing} and \eqref{eq:dcondy}. \\
		Let now $p_j = \n{\pi(T_{\OM}(y_j), -\nabla f(y_j))}$ and $\tilde{p}_j = \n{\pi(T_{\OM}(y_j), g)} = \n{\pi(T_{\OM}(y_j), -\nabla f(x_k))}$. We have
		\begin{equation} \label{projineq}
			\begin{aligned}
				|p_j - \tilde{p}_j| & = |\n{\pi(T_{\OM}(y_j), -\nabla f(y_j))} - \n{\pi(T_{\OM}(y_j), -\nabla f(x_k))}| \\ &\leq \n{-\nabla f(y_j) + \nabla f(x_k)} \leq L\n{y_j-x_k} \, ,
			\end{aligned}		
		\end{equation}
		where we used the 1-Lipschitzianity of the projections on convex sets in the first inequality. We now distinguish four cases according to how the SSC terminates. \\
		\textbf{Case 1:} $T = 0$ or $d_T = 0$. Since there are no descent directions $x_{k +  1}^{tr} = y_T$ must be stationary for the gradient $g$.  Equivalently, $\tilde{p}_T = \n{\pi(T_{\OM}(x_{k +  1}^{tr}), g)} = 0$. We can now write
		\begin{equation*}
			\n{x_{k +  1}^{tr}-x_k} \geq \frac{1}{L}(|p_T - \tilde{p}_T|) = \frac{p_T}{L} > \K p_T \, ,
		\end{equation*}
		where we used \eqref{projineq} in the first inequality and $\tilde{p}_T = 0$ in the equality. We now prove that if $T = 0$ then necessarily $d_0 = 0$. This is clear if $y_0$ is stationary for $-g$. Otherwise $d_0 \neq 0$ is a feasible descent direction, so that $\alpha^{(0)}_{\max} > 0$ and also $\beta_0 > 0$. But then $\alpha_0 > 0$ and the SSC can't terminate in Phase I or Phase II, in contradiction with $T = 0$.  \\
		Before examining the remaining cases we remark that if the SSC terminates in Phase III then $\alpha_{T- 1} = \beta_{T-1}$ must be maximal w.r.t. the conditions $y_T \in B_{T-1}$ or $y_T \in \bar{B}$, which imply $y_T \in \partial B_{T - 1}$ (case 3) and $y_T \in \partial \bar{B}$ (case 4) respectively. On the other hand if $\alpha_T = 0$ then necessarily $y_T \in \tx{int}(\OM_T)^c$. In this case we cannot have $y_T \in \partial \bar{B}$, otherwise the SSC would terminate in Phase III of the previous cycle. Therefore necessarily $y_T \in \tx{int}(B_T)^c$ (Case 2). \\ 
		\textbf{Case 2:} $y_T \in \tx{int}(B_T)^c$. We can rewrite the condition as
		\begin{equation} \label{trueT}
			\s{g}{\hat{d}_T} \leq L\n{y_T - x_k} \, .
		\end{equation}
		Thus 
		\begin{equation} \label{t2T}
			p_T \leq \tilde{p}_T + L\n{y_T - x_k} \leq \frac{1}{\tau}\s{g}{\hat{d}_T} + L\n{y_T - x_k} \leq \left(\frac{L}{\tau} + L\right) \n{y_T - x_k} \, ,
		\end{equation}
		where the first inequality follows from \eqref{projineq}, the second from $\frac{\s{g}{\hat{d}_T}}{\tilde{p}_T} \geq \DSB_{\mathcal{A}}(\OM, y_{T}, g) \geq \SB_{\mathcal{A}}(\OM) = \tau $, and the third from \eqref{trueT}. Then $\tilde{x}_k = x_{k + 1}^{tr} = y_T$ satisfies the desired conditions. \\ 		
		\textbf{Case 3:} $y_T = y_{T - 1} + \beta_{T - 1} d_{T-1}$ and $y_T \in \partial B_{T-1}$. Then from $y_{T-1} \in B_{T-1}$ it follows
		\begin{equation} \label{c3y}
			L \n{y_{T-1} - x_k} \leq \s{g}{\hat{d}_{T-1}} \, ,
		\end{equation} 
		and $y_T \in \partial B_{T-1}$ implies
		\begin{equation}\label{t1T}
			\s{g}{\hat{d}_{T-1}} = L \n{y_T - x_k} \, .
		\end{equation}
		Combining \eqref{c3y} with \eqref{t1T} we obtain 
		\begin{equation} \label{yt1}
			L \n{y_{T - 1} - x_k} \leq L \n{y_T - x_k} \, .
		\end{equation}
		Thus 
		\begin{equation*}
			p_{T - 1} \leq \tilde{p}_{T - 1} + L\n{y_{T - 1} - x_k} \leq \frac{1}{\tau}\s{g}{\hat{d}_{T - 1}} + L\n{y_{T - 1} - x_k} \leq \left(\frac{L}{\tau} + L\right) \n{y_T - x_k}	\, ,	
		\end{equation*}
		where we used \eqref{t1T}, \eqref{yt1} in the last inequality and the rest follows reasoning as for \eqref{t2T}. In particular we can take $\tilde{x}_k = y_{T-1}$. \\
		\textbf{Case 4:} $y_T = y_{T - 1} + \beta_{T - 1} d_{T-1}$ and $y_T \in \partial \bar{B}$. \\
		By the characterization of $\bar{B}$ in \eqref{barBeq} the condition $x_{k +  1}^{tr} = y_T \in \bar{B}$ can be rewritten as
		\begin{equation} \label{case2c}
			L\n{x_{k +  1}^{tr} - x_k}^2 - \s{g}{x_{k +  1}^{tr} - x_k} = 0 \, .
		\end{equation}
		For every $j \in \{0, ..., T\}$ we have
		\begin{equation} \label{eq:rec}
			x_{k +  1}^{tr} = y_j + \sum_{i=j}^{T-1} \alpha_i d_i \, .
		\end{equation}
		We now want to prove that for every $j \in \{0, ..., T\}$ 
		\begin{equation} \label{eq:claim2}
			\n{ x_{k +  1}^{tr} - x_k} \geq \n{y_j - x_k} \, .
		\end{equation}
		Indeed, we have
		\begin{equation*}
			\begin{aligned}
				L\n{ x_{k +  1}^{tr} - x_k}^2 = \s{g}{x_{k +  1}^{tr} - x_k} &= \s{g}{y_j - x_k} + \sum_{i=j}^{T-1} \alpha_i \s{g}{d_i} \\ &\geq \s{g}{y_j - x_k} \geq L\n{y_j - x_k}^2 \, ,
			\end{aligned}
		\end{equation*}
		where we used \eqref{case2c} in the first equality, \eqref{eq:rec} in the second, $\s{g}{d_j} \geq 0$ for every $j$ in the first inequality and $y_j \in \bar{B}$ in the second inequality. \\
		We also have 
		\begin{equation} \label{eq:fracineq}
			\begin{aligned}
				\frac{\s{g}{x_{k +  1}^{tr} - x_k}}{\n{x_{k +  1}^{tr} - x_k}}  = \frac{\s{g}{\sum_{j=0}^{T-1}\alpha_j d_j}}{\n{\sum_{j=0}^{T-1}\alpha_j d_j}} & \geq
				\frac{\s{g}{\sum_{j=0}^{T-1}\alpha_j d_j}}{\sum_{j=0}^{T-1}\alpha_j \n{d_j}} \\
				\geq \min \left\{\frac{\s{g}{d_j}}{\n{d_j}} \ | \ 0  \leq j \leq T-1 \right\} \, .
			\end{aligned}
		\end{equation}
		Thus for $\tilde{T} \in \argmin \left\{ \frac{\s{g}{d_j}}{\n{d_j}} \ | \ 0  \leq j \leq T-1 \right\}$ 
		\begin{equation} \label{tildeT}
			\s{g}{\hat{d}_{\tilde{T}}} \leq 	\frac{\s{g}{x_{k +  1}^{tr} - x_k}}{\n{x_{k +  1}^{tr} - x_k}} = L\n{x_{k +  1}^{tr} - x_k} \, ,
		\end{equation}
		where we used \eqref{eq:fracineq} in the first inequality and \eqref{case2c} in the second. \\
		We finally have
		\begin{equation*}
			p_{\tilde{T}} \leq \tilde{p}_{\tilde{T}} + L\n{y_{\tilde{T}} - x_k} \leq \frac{1}{\tau}\s{g}{\hat{d}_{\tilde{T}}} + L\n{y_{ \tilde{T}} - x_k} \leq \left(\frac{L}{\tau} + L\right) \n{x_{k +  1}^{tr} - x_k} \, ,
		\end{equation*}
		where we used \eqref{eq:claim2}, \eqref{tildeT} in the last inequality and the rest follows reasoning as for \eqref{t2T}. In particular $\tilde{x}_k = y_{\tilde{T}}$ satisfies the desired properties. 	\\ 
		It remains to prove that $x_{k + 1}$ satisfies \eqref{eq:H1} and \eqref{d'2}. The first condition \eqref{eq:H1} clearly holds by \eqref{eq:step4}. As for \eqref{d'2}, we have 
		\begin{equation*} 
			f(x_k) - f(x_{k + 1}) \geq f(x_{k} - f(x_{k + 1}^{tr}) \geq \frac{L}{2} \n{x_k - x_{k + 1}^{tr}}^2 \geq \frac{2\K ^2}{L} \n{\pi(T_{\OM}(\tilde{x}_k), -\nabla f(\tilde{x}_k))}^2 \, ,
		\end{equation*}
		where the first inequality follows from \eqref{eq:step4}, the second from \eqref{dq}, and the third from \eqref{dq2}. Finally, since $f(x_{k + 1}) \leq f(x_{k + 1}^{tr})$, by \eqref{xktildecond} the conditions \eqref{eq:gd''2} on $\tilde{x}_k$ are satisfied. 
	\end{proof}
	
	As an immediate corollary we have convergence to the set of stationary points with a rate of $O(\frac{1}{\sqrt{k}})$ for the norm of the projected gradient $\pi(T_{\OM}(\tilde{x}_i), - \nabla f (\tilde{x}_i))$. While the norm of the projected gradient may not be available, by \eqref{dq} we have the readily available upper bound $\n{x_{k} - x_{k +  1}^{tr}}/\K$, which has an analogous convergence rate. 
	\begin{corollary} \label{cor:stationary}
		Under the assumptions of Proposition \ref{keyprop}, $\{f(x_k)\}$ is decreasing, $f(x_k) \rightarrow f^* \in \mathbb{R}$ and the limit points of $\{x_k\}$ are stationary. Furthermore, for any sequence $\{\tilde{x}_k\}$ satisfying \eqref{dq2},
		\begin{equation} \label{eq:qsqrk}
			\min_{0 \leq i \leq k} \n{\pi(T_{\OM}(\tilde{x}_i), - \nabla f (\tilde{x}_i))} \leq \min_{0 \leq i \leq k} \frac{\n{x^{tr}_{i + 1} - x_i}}{\K} \leq  \sqrt{ \frac{2(f(x_0) - \tilde{f})}{\K ^2L(k + 1)}} \, .
		\end{equation}
	\end{corollary}
	\begin{proof}
		The sequence $\{f(x_k)\}$ is decreasing by \eqref{eq:step4} and \eqref{dq}  . Thus by compactness $f(x_k) \rightarrow \tilde{f} \in \R$ and in particular $f(x_k) - f(x_{k + 1}) \rightarrow 0$. Since $f(x_{k + 1}^{tr}) - f(x_k) \leq f(x_k) - f(x_{k + 1})$ we also have $f(x_{k + 1}^{tr}) - f(x_k) \rightarrow 0$ and by \eqref{dq} $x_{k + 1}^{tr} - x_k \rightarrow 0$.  Let $\{x_{k(i)}\} \rightarrow \tilde{x}^*$ be any convergent subsequence of $\{x_k\}$. For $\{\tilde{x}_k\}$ chosen as in the proof of Proposition \ref{keyprop} we have $\n{\tilde{x}_k - x_k} \leq \n{x_k^{tr} - x_k}$ because $\tilde{x}_k = y_T = x_k^{tr}$ in case 1 and case 2, by \eqref{yt1} in case 3, and by \eqref{eq:claim2} in case 4. Therefore
		$$\n{\tilde{x}_{k(i)} - x_{k(i)}} \leq \n{x_{k(i)} - x_{k(i) + 1}^{tr}} \rightarrow 0 \, .$$ Furthermore,  $\n{\pi(T_{\OM}(\tilde{x}_{k(i)}), -\nabla f(\tilde{x}_{k(i)})))}~\leq~\frac{\n{x_{k(i)} - x_{k(i) + 1}^{tr}}}{K} \rightarrow 0 $ again by Proposition \ref{keyprop}, so that $\tilde{x}_{k(i)} \rightarrow \tilde{x}^*$ with $\n{\pi(T_{\OM}(\tilde{x}_{k(i)}), -\nabla f(\tilde{x}_{k(i)}))} \rightarrow 0$. Then $\n{\pi(T_{\OM}(\tilde{x}^*), -\nabla f(\tilde{x}^*))} =0 $ and $\tilde{x}^*$ is stationary. 
		
		The first inequality in \eqref{eq:qsqrk} follows directly from \eqref{dq2}. As for the second, we have 
		\begin{equation*} 
			\begin{aligned}
				& \frac{k + 1}{\K ^2} (\min_{0 \leq i \leq k} \n{x_{i + 1}^{tr} - x_i})^2 = \frac{k + 1}{\K ^2} \min_{0 \leq i \leq k} \n{x_{i + 1}^{tr} - x_i}^2  \\
				\leq &\frac{1}{\K ^2} \sum_{i= 0}^k \n{x_{i} - x_{i + 1}^{tr}}^2 \leq \frac{2}{L\K ^2} \sum_{i = 0}^{k}(f(x_{i + 1}) - f(x_i)) \leq \frac{2(f(x_0) - \tilde{f})}{L\K ^2} \, ,
			\end{aligned}
		\end{equation*}
		where we used \eqref{dq} together with \eqref{eq:step4} in the second inequality, $f(x_i) \rightarrow \tilde{f}$ decreasing in the second and \eqref{eq:qsqrk} follows by rearranging. 
	\end{proof}
	\begin{remark}
		The proof of Proposition gives an easy way to retrieve $\tilde{x}_k$. In particular it shows that one can take $\tilde{x}_k = y_{\tilde{T}}$ with $\tilde{T} = T$ in case 1 and case 2, $\tilde{T} = T - 1$ in case 3 and 
		\begin{equation*}
			\tilde{T} \in \argmin \{\s{g}{\hat{d}_i} \ | \ i \in [0: T] \}		
		\end{equation*} 
		in case 4. 
	\end{remark}
	\begin{remark}
		For any $d \in \mathcal{A}(\tilde{x}_k, - \nabla f(\tilde{x}_k))$ we have
		\begin{equation*}
			\s{-\nabla f(\tilde{x}_k)}{\hat{d}} \leq	\n{\pi(T_{\OM}(\tilde{x}_i), - \nabla f (\tilde{x}_i))} \leq \frac{\s{ -\nabla f(\tilde{x}_k)}{\hat{d}}}{\tau} \, ,
		\end{equation*}
		where the first inequality follows from Proposition \ref{NWessential} and the second by the assumption on $\tx{SB}_{\mathcal{A}}(\OM)$. Therefore at the price of computing $- \nabla f(\tilde{x}_k)$ and a related descent directions $d$ we have upper and lower bounds separated by a factor of $\tau$ for $\n{\pi(T_{\OM}(\tilde{x}_i), - \nabla f (\tilde{x}_i))}$.
	\end{remark}
	
	\section{Convergence properties} \label{s:CP}
	In this section, we use the descent sequence properties proved in the previous section in combination with a KL property for smooth functions to obtain convergence rates with respect to the objectives and the tail length of the iterate sequence. First, we give an abstract convergence lemma for proper lower semicontinuous functions. We then apply this lemma in combination with Proposition \ref{keyprop} to prove convergence rates for Algorithm \ref{tab:3} under finite termination of the SSC, and thus in particular for the examples discussed in Section \ref{s:examples}. 
 
	For the abstract lemma with respect to the foundational work \cite{attouch2013convergence} we only have the weaker condition \eqref{eq:gd'2} instead of \eqref{eq:H2}, where the proof of the convergence lemma \cite[Lemma 2.6]{attouch2013convergence} relies on the stronger condition. We therefore present a different proof technique  in Section \ref{sigmaalpha} of the appendix based on Karamata's inequality (\cite{kadelburg2005inequalities}, \cite{karamata1932inegalite}).
	
	Before stating the main convergence result we need to introduce a one dimensional worst case sequence related to the desingularizing function $\varphi$, which we use bound the error on the objective. With respect to the (different) worst case sequence introduced in \cite{bolte2017error} we do not require neither the moderate behavior hypothesis on $\varphi$ nor $(\varphi')^{-1}$ to be Lipschitz regular. 
	
	Let $\alpha: [0, \eta) \rightarrow \R_{> 0} \cup\{ \infty \}$ be a continuous monotone non increasing function finite in $(0, \eta)$. For $t \in [0, \eta)$ we define 
	\begin{equation*}
		\sigma_{\alpha}(t) = \max \left\{s \in [0, t] \ | \ \frac{1}{\alpha(s)^2} \leq t - s \right\}  \, ,
	\end{equation*}
	with the convention $\max (\emptyset) =  0$. Since $s + 1/\alpha(s)^2$ is continuous and strictly increasing if
	$$ \left\{ s \in [0, t] \ | \ \frac{1}{\alpha(s)^2} \leq t - s \right\}$$ is non empty then it has a maximum and 
	\begin{equation} \label{s2eq}
		t - \sa(t) = \frac{1}{\alpha(\sa(t))^2} \, .
	\end{equation}
	Otherwise 
	\begin{equation} \label{0sa}
		t - \sa(t)  = t \leq \frac{1}{\alpha(\sa(t))^2} = \frac{1}{\alpha(0)^2} \, .
	\end{equation}
	Let $\sigma_{\alpha}^{(k)}$ be equal to $\sigma_{\alpha}$ applied $k$ times, with the convention $\sigma_{\alpha}^{(0)}(t) = t$. Then the worst sequence related to an initial error $f_0 \in [0, \eta)$ is $\{\sigma_{\alpha}^{(k)}(f_0)\}$. Given that by continuity $\alpha$ is upper bounded in every interval $[\varepsilon, \eta)$ with $\varepsilon > 0$, we have the limit
	\begin{equation} \label{eq:tsigmalim}
		\lim_{k \rightarrow \infty} \sigma_{\alpha}^{(k)}(f_0) \rightarrow 0
	\end{equation}
	for every $f_0 \in [0, \eta)$.
	
	We can now state the main convergence result for proper lower semicontinuous functions. 
	
	\begin{proposition} \label{keytec}
		Let $\f: \R^n \rightarrow \R \cup \{\infty\}$ be a proper l.s.c. function satisfying the KL inequality w.r.t. $x^*$ in $B_{\delta}(x^*) \cap [\f(x^*) < \f < \eta]$. Let then $\{x_k\}$ be a sequence satisfying \eqref{eq:H1} and \eqref{eq:gd'2} such that 	
		\begin{align} 
			\f(x^*) &\leq \f(x_0) < \eta \label{n4fx0fx*} \, , \\
			x_k \in B_{\delta}(x^*) &\Rightarrow \f(x_{k+1}) \geq \f(x^*) \, . \label{cond:boundc}
		\end{align}
		Let $\alpha(t) = \frac{b}{\sqrt{a}} \varphi'(t)$. Assume also that
		\begin{equation} \label{corcond}
			\frac{b}{a}\varphi(\f(x_0) - \f(x^*)) + 2 \sqrt{\frac{\f(x_0) - \f(x^*) -  \sigma_{\alpha}(\f(x_0) - \f(x^*))}{a}} + \n{x_0 - x^*} < \delta \, .
		\end{equation} 
		Then $ B_{\delta}(x^*) \supset \{x_k\} \rightarrow \tilde{x}^*$ with $f(\tilde{x}^*) \leq f(x^*)$, and 
		\begin{align}
			\f(x_k) -\f(x^*) \leq & \sigma^{(k)}_{\alpha}(\f(x_0) - \f(x^*)) \, , \label{convergencerates2}  \\ 
			\n{x_{k} - \tilde{x}^*} \leq \sum_{i= k}^{\infty} \n{x_i - x_{i+1}}  \leq & \frac{b}{a}\varphi(\f(x_k) - \f(x^*)) + 2 \sqrt{\frac{\f(x_k) - \f(x^*) -  \sigma_{\alpha}(\f(x_k) - \f(x^*))}{a}} \, . \label{convergencerates3}
		\end{align}	
		
	\end{proposition}
	Condition \eqref{corcond} is a slight refinement of \cite[condition (4)]{attouch2013convergence}. It implies of course $x_0 \in B_\delta(x^*)$, and the additional terms on the LHS are sufficient to ensure $\{x_k\}, \{\tilde{x}_k\} \subset B_{\delta}(x^*)$. It is not difficult to see from the proofs that it can be replaced with $\{x_k\}, \{\tilde{x}_k\} \subset B_{\delta}(x^*)$.  	
	When the desingularizing function is of the form $\varphi(t) = \frac{M}{\theta} t^{\theta}$ with $\theta \in (0, \frac{1}{2}]$ we obtain more explicit convergence rates on the length and the objective gap. For convex functions this KL property is equivalent to H\"{o}lderian error bounds, as proved more in general in \cite[Theorem 5]{bolte2017error}. 
	
	\begin{corollary} \label{holderian}
		Under the assumptions of Proposition \ref{keytec}, let $\varphi(t) = \frac{M}{\theta}t^{\theta}$ for some $\theta \in (0, \frac{1}{2}]$. Then 
		\begin{equation} \label{thetafconv}
			\f(x_k) - \f(x^*) = 
			\begin{cases}
				O(1/k^{\frac{1}{1-2\theta}}) &\text{ for } 0 < \theta < \frac{1}{2} \, , \\
				O((1 + \frac{a}{b^2M^2})^{-k})  & \text{ for } \theta = \frac{1}{2} \, .
			\end{cases}
		\end{equation}
		Moreover
		\begin{equation} \label{thetatailconv}
			\sum_{i = k}^\infty \n{x_{i} - x_{i+1}} = 
			\begin{cases}
				O(1/k^{\frac{\theta}{1-2\theta}}) &\text{ for } 0 < \theta < \frac{1}{2} \, , \\
				O((1 + \frac{a}{b^2M^2})^{-\frac{k}{2}}) &\text{ for } \theta = \frac{1}{2} \, .
			\end{cases}
		\end{equation}
		Finally, the implicit constants can be taken as functions of $M, \theta, a, b, \eta$, and in particular independent from $x_0$.
	\end{corollary} 
	The proofs of Proposition \ref{keytec} and Corollary \ref{holderian} are included in Section \ref{sigmaalpha}.

	As a corollary of Proposition \ref{keytec}, by Proposition \ref{keyprop} we have the following results on the convergence rates of Algorithm \ref{tab:3}: 
	\begin{corollary} \label{cor:corA2conv}
		Let us consider Problem \eqref{eq:mainpb} with $f \in C^1(\OM)$ satisfying the KL inequality w.r.t. $x^*\in \Omega$ in $B_{\delta}(x^*) \cap [f(x^*) < f < \eta]$. Assume that $\{x_k\}$ is generated by Algorithm \ref{tab:3}, and that the method $\mathcal{A}$  in the SSC satisfies the following conditions:
		\begin{itemize}
			\item  the SSC  procedure always terminates in a finite number of steps;
			\item $\SB_{\mathcal{A}}(\OM)=\tau>0$.
		\end{itemize}  Assume that \eqref{n4fx0fx*}, \eqref{cond:boundc} and \eqref{corcond} hold with $a = \frac{L}{2}$ and $b = \tau/(L(1 + \tau))$.  Then the sequence converges to a stationary point and the convergence rates \eqref{convergencerates2} and \eqref{convergencerates3} hold. 
	\end{corollary}
	\begin{proof}
		Under our general assumptions on $\mathcal{A}$ by Proposition \ref{keyprop} the sequence $\{x_k\}$ satisfies \eqref{eq:H1} and \eqref{eq:gd'2} in the form \eqref{d'2}. It then suffices to apply Proposition \ref{keytec} to $\f = f + i_{\OM}$ to obtain the convergence rates \eqref{convergencerates2}, \eqref{convergencerates3}. Finally, $\tilde{x}^*$ is stationary by Corollary \ref{cor:stationary} and $f(\tilde{x}^*) = f(x^*)$ by continuity. 
	\end{proof}
	
	\begin{remark}
		If $x_{k + 1} = x_{k + 1}^{tr}$ in Algorithm \ref{tab:3} the property \eqref{cond:boundc} holds in particular when $x^*$ is a global minimum in the connected component $C_0$ containing $x_k$ of the sublevel set $[- \infty < f \leq f(x_k) ]$. Indeed $x_{k + 1}^{tr}$ cannot be outside $C_0$ since
		\begin{equation*}
			x^{tr}_{k+1} \in B_{\frac{\n{g}}{2L}}(x_k + \frac{g}{2L}) \subset C_0 \, ,
		\end{equation*}
		where $g = -\nabla f(x_k)$ and  $B_{\frac{\n{g}}{2L}}(x_k + \frac{g}{2L}) \subset C_0$ because it is connected and contained in $$[- \infty < f \leq f(x_k) ] $$ by \eqref{stdbarB}.
	\end{remark}

	\section{Analysis of some projection-free first-order methods}\label{s:examples}
	Here, we report some relevant examples of projection-free first-order methods and show that 
	those methods satisfy the assumptions given in Corollary \ref{cor:corA2conv}. More specifically, we show finiteness of the SSC procedure and that a sufficient slope condition is satisfied when the method is applied over a specific class of feasible sets. The reported examples include the AFW, PFW, FDFW on polytopes, the FDFW on sublevel sets of strongly convex smooth functions, a method based on orthographic retractions for convex sets with smooth boundary, as well as combination of these methods on product domains. At the end of the section we show how to apply our framework together with a KL property of (non convex) quadratic programming problems to obtain an asymptotic linear convergence rate for some FW variants. 
	
	Before giving the examples we prove a broad SSC finite termination criterion, showing that the procedure ends when using a first order method with mild convergence properties for linear objectives.

	\begin{lemma} \label{finiteTermCrit}
		Assume that the method $\bA$ applied to any linear function $L_g(x) = -\s{g}{x}$ and with every stepsize maximal always generates a (possibly finite) sequence $\{y_j\}$ such that 
		\begin{equation} \label{convergence_assumption}
			\liminf \pi_{y_j}(g) = 0 \, .
		\end{equation}
		Then the SSC with the method $\bA$ always terminates in a finite number of steps. 
	\end{lemma}
\begin{proof}
	Assume by contradiction that the SSC generates an infinite sequence $\{y_j\}$.
 In this case the method $\bA$ applied in the SSC with gradient $-g$ always does maximal steps, since the SSC terminates as soon as $\alpha_j < \alpha_{\max}^{(j)}$ so that $\alpha_j = \beta_j$. Let $p_{j} = \pi_{y_{j}}(g)$. Then by \eqref{convergence_assumption} we can take $\{j(k)\}$ subsequence of indexes such that $p_{j(k)} \rightarrow 0$.	We claim that $y_{j(k)} \rightarrow y_0$, in contradiction with the strict monotonicity of $j \rightarrow \s{g}{y_j}$. Indeed we have 
	$$\n{y_{j(k)} - y_0} \leq 
	\frac{\s{\hat{d}_{j(k)}}{g}}{L} \leq \frac{p_{j(k)}}{L} \rightarrow 0 \, ,$$
	by definition of $\OM_{j(k)}$ in the first inequality and by Proposition \ref{NWessential} in the second. 
\end{proof}

	\subsection{PFW, AFW, FDFW directions} \label{pfwafw}
	The AFW and PFW depend from a set of "elementary atoms" $A$ such that $\OM = \conv(A)$. Given $A$, for a base point $x \in \OM$ we can define
	$$S_x = \{S \subset A \ | \ x \tx{ is a proper convex combination of all the elements in }S \} \, ,$$ the family of possible active sets for $x$. In the rest of the article $A$ is always clear from the context and for simplicity we write $\tx{PFW}$, $\tx{AFW}$ instead of $\tx{PFW}_A$, $\tx{AFW}_A$. For $x \in \OM$,  $S \in S_x$, $d^{\tx{PFW}}$ is a PFW direction with respect to the active set $S$ and gradient $-g$ iff
	\begin{equation*} 
		d^{\tx{PFW}} = s - q \textnormal{ with } s \in \argmax_{s \in \OM} \s{s}{g} \textnormal{ and } q \in \argmin_{q \in S} \s{q}{g} \, .
	\end{equation*}
	Similarly, given $x \in \OM$, $S \in S_x$, $d^{\tx{AFW}}$ is an AFW direction with respect to the active set $S$ and gradient $-g$ iff
	\begin{equation} \label{AFWdir}
		d^{\tx{AFW}} \in \argmax \{\s{g}{d} \ | \ d \in \{d^{\tx{FW}}, x-q\} \} \textnormal{ with } q \in \argmin_{q \in S} \s{q}{g} \, ,
	\end{equation}
	where $d^{\tx{FW}}$ is a classic Frank Wolfe direction
	\begin{equation} \label{eq:FWstep}
		d^{\tx{FW}}  = s - x  \textnormal{ with } s \in \argmax_{s \in \OM} \s{s}{g} \, .
	\end{equation}
If the FW method is applied to a linear objective $L_g(x) = -\s{g}{x}$ it clearly terminates after at most one maximal step, with $y_1 \in \argmin_{x \in \OM} L_g(x)$. As for the AFW or the PFW, they generate a sequence of active sets $\{S^{(j)}\}$ with $y_j$ proper convex combination of the elements in $S^{(j)}$. After a FW step or a PFW step the linear minimizer $s_j$ is always in the active set, and we assume that no other point can be added at the step $j$. In particular, $S^{( j + 1)} \subset S^{(j)} \cup \{s_j\}$. Now on the one hand, all the linear minimizers added to the active set can't be dropped from the active set, since the gradient doesn't change. On the other hand, after every maximal away or PFW step the element $q_j$ corresponding to the away direction is dropped from the active set. Furthermore, as the FW method also the AFW terminates after a maximal FW step. Then the AFW and the PFW applied to linear objectives terminate in at most $|S^{(0)}|$ steps. Applying Lemma \ref{finiteTermCrit} to obtain finite SSC termination we have the following:
	\begin{proposition} \label{AFWPFWfiniteTerm}
		For the FW, the PFW and the AFW applied to a linear objective the number of maximal steps is bounded, provided that no points beside linear minimizers are added to the active set. In particular, under this assumption for these methods the SSC terminates.  
	\end{proposition}	

	The FDFW from \cite{freund2017extended} relies only on the current point $x$ and the current gradient $-g$ to choose a descent direction and, unlike the AFW and the PFW, does not need to keep track of the active set. This method makes use of a linear minimization oracle 
\begin{equation*}
\tx{LMO}_C(-g) \in \textnormal{argmin}_{x \in C}\s{-g}{x}\, ,
\end{equation*}
for $C$ varying among the faces of $\OM$. 
Let $\F(x)$ be the minimal face of $\OM$ containing $x$. The in face direction is defined as
\begin{equation*}
d^A = x_k - x_A \tx{ with } x_A \in \argmin\{\s{g}{y} \ | \ y \in \F(x) \} \, .
\end{equation*}
The selection criterion is then analogous to the one used by the AFW:
\begin{equation} \label{crit:FD}
d^{\tx{FD}} \in \tx{argmax} \{ \s{g}{d} \ | \ d \in \{d^A, d^{\tx{FW}} \} \} \, .
\end{equation}
For the FDFW we assume that the maximal stepsize is given by feasibility conditions as in \cite{freund2017extended}:
\begin{equation} \label{eq:alphamax}
\alpha_{\max}(x, d) = \max \{\alpha \in \R_{\geq 0} \ | \ x + \alpha d \in \OM\} \, .
\end{equation}
Then after a maximal in face step from $y_j$ we have $\tx{dim} \mathcal{F}(y_{j+1}) < \tx{dim} \mathcal{F}(y_{j})$ because $y_{j+1}$ lies on the boundary of $\mathcal{F}(y_j)$. Whence there can only be a finite number of consecutive such steps. Furthermore, we've seen that after a maximal FW step every method applied to a linear objective terminates. Applying Lemma \ref{finiteTermCrit} to obtain finite SSC termination we have the following:
\begin{proposition} \label{FDFWfiniteterm}
	The FDFW on any compact and convex set $\OM$ does at most $\tx{dim}(\OM) + 1$ consecutive maximal steps. In particular, for this method the SSC terminates.
\end{proposition}
We write $\SB_{\tx{FD}}, \DSB_{\tx{FD}}$ instead of $\SB_{\tx{FDFW}}, \DSB_{\tx{FDFW}}$ in the rest of the paper.

When $\OM$ is a polytope and $|A| < \infty$ the sufficient slope conditions hold for the directions we introduced. Before stating a lower bound for $\tx{SB}_{\A}(\OM)$ in this setting we need to recall the definition of pyramidal width $\PWidth(A)$ as it was given in \cite{lacoste2015global}. We refer the reader to \cite{pena2018polytope} for a discussion of various properties of this parameter. 

For a given $g \in \mathbb{R}^n \sm \{0\}$ the pyramidal directional width is defined as 
	\begin{equation} \label{def:PdirW}
		\tx{PdirW}(A, g, x) = \min_{S \in S_x} \max_{\substack{a \in A \\ s \in S}} \s{\frac{g}{\n{g}}}{a-s} \, ,
	\end{equation}
	and the pyramidal width is defined as 
	\begin{equation*}
		\tx{PWidth}(A) = \min_{\substack{F \in \tx{faces}(\conv(A)), \, x \in F \\ g \in \tx{cone}(F-x) \sm \{0\} }} \tx{PdirW}(F\cap A, g, x) \, .
	\end{equation*} 
	Here we use one key property of $\tx{PWidth}(A)$ which relates it to the slope along the PFW direction. 
	We have the following lower bound (see \cite[equation (12)]{lacoste2015global}) 
	\begin{equation} \label{eq:JJ}
		\frac{\s{g}{d^{\tx{PFW}}}}{\s{g}{\hat{e}}} \geq \textnormal{PWidth}(A) > 0 \, ,
	\end{equation}
	where $e$ is any direction in $\OM - \{x\} = T_{\OM}(x)$ which is also a descent direction\footnote{In \cite{lacoste2015global} the direction $e$ is defined as a possible error direction $e=x^* - x$, where $x^*$ is an optimal point of a convex objective $f$ with $\nabla f(x)= -g$. However, this definition is equivalent to ours. Indeed if $e = x^* - x$ then by convexity it must be a feasible descent direction for $-g$. Conversely, every feasible descent direction is always an error direction as defined above for some choice of $f$, i.e. consider $f(y) =\frac{1}{2}\s{y-x^*}{Q(y-x^*)}$ with $Q$ positive definite such that $\nabla f(x) = Q(x-x^*)=-g$.}  for $-g$.
 Another relevant property is that by \eqref{def:PdirW} we have that $\tx{PWidth}(A)$ is monotone decreasing in $A$, so that if $V(\OM)$ is the set of vertexes of $\OM$ we always have $\tx{PWidth}(A) \leq \tx{PWidth}(V(\OM))$. 
	\begin{proposition}\label{eqphi}
		Let $\textnormal{diam}(\Omega) = D$. Then 
		\begin{equation} \label{PFWAFW}
			\begin{aligned}
				\textnormal{SB}_{\tx{PFW}}(\Omega) &\geq \frac{\PWidth(A)}{D} \, , \\
				\textnormal{SB}_{\tx{AFW}}(\Omega) &\geq \frac{\PWidth(A)}{2D} \, , \\
				\textnormal{SB}_{\tx{FD}}(\Omega) &\geq \frac{\PWidth(V(\OM))}{2D} \, .
			\end{aligned}
		\end{equation}		
	\end{proposition} 
	\begin{proof}
		Let $g$ be such that $\pi_x(g) \neq 0$. Then there exists descent directions for $-g$ feasible for $\OM$ from $x$, and $$0 < \max_{e \in \OM - \{x\}} \s{g}{\hat{e}} \, , $$ so that
		\begin{equation} \label{preliminary}
			\min_{ \substack{e \in \OM - \{x\},  \\ \s{g}{\hat{e}} > 0}} \frac{1}{\s{g}{\hat{e}}} = \frac{1}{\max_{e \in \OM - \{x\}} \s{g}{\hat{e}}} 		 =   \frac{1}{	\sup_{h\in\Omega \setminus \{\bar{x}\}} \left (g, \frac{h-\bar{x}}{\|h - \bar{x}\|}\right ) } = \frac{1}{\n{\pi(T_{\OM}(x), g)}} \, ,
		\end{equation}
		where we used Proposition \ref{NWessential} in the last equality. Thanks to \eqref{preliminary} taking the min on all the feasible descent directions $e$ in the LHS of \eqref{eq:JJ} we obtain  
		\begin{equation*} 
			\frac{\s{g}{d^{\tx{PFW}}}}{\n{\pi(T_{\OM}(x), g)}} \geq \textnormal{PWidth}(A) \, .
		\end{equation*}
		We now have
		\begin{equation*}
			\DSB_{\tx{PFW}}(\OM, x, g) = \inf_{d^\tx{PFW} \in \tx{PFW}(x, g)} \frac{\s{g}{d^{\tx{PFW}}}}{\n{d^{\tx{PFW}}}\n{\pi(T_{\OM}(x), g)}} \geq \frac{\s{g}{d^{\tx{PFW}}}}{D\n{\pi(T_{\OM}(x), g)}} \geq \frac{\textnormal{PWidth}(A)}{D} \, .
		\end{equation*}
		and the first part of \eqref{PFWAFW} follows by taking the inf on the LHS for $x \in \OM$ and $g$ such that $\pi_x(g) \neq 0$.
		Inequality \eqref{PFWAFW} for the AFW method is a corollary since
		\begin{equation*}
			\s{g}{d^{\tx{AFW}}} \geq \frac{1}{2}\s{g}{d^{\tx{PFW}}} \, ,
		\end{equation*}
		as it follows immediately from the definitions (see also \cite[equation (6)]{lacoste2015global}). 
		
			For the FDFW, first observe that for any $S \in S_x$:
		\begin{equation} \label{eq:gdjP}
		\begin{aligned}
	    	\s{g}{d^{\tx{FW}} + d^A} = \max\{\s{g}{ s - q } \ | \ s \in \OM, \ q \in \F(x) \} \geq &
		\max\{\s{g}{s-q} \ | \ s \in \OM, \ q \in S \} \\
		= &  \s{g}{d^{\tx{PFW}}} \, ,
		\end{aligned}		
		\end{equation}
		where we used $\F(x) \supset S$ in the inequality and the last equality follows immediately from the definition of $d^{\tx{PFW}}$. We then have
		\begin{equation} \label{DSBineqFD} 
		\tx{DSB}_{\tx{FD}}(\OM, x, g) = \frac{\s{g}{d^{\tx{FD}}}}{\pi_x(g) \n{d^{\tx{FD}}}}  \geq \frac{\s{g}{d^{\tx{FW}} + d^A }}{2D\pi_x(g)} \geq  \frac{ \s{g}{d^{\tx{PFW}}}}{2\pi_x(g)\n{d^{\tx{FD}}}} \geq \frac{\textnormal{PWidth}(A)}{2D} \, ,
		\end{equation}
		where the first inequality follows from the choice criterion \eqref{crit:FD} together with $\n{d^{\tx{FD}}} \leq D$, we used \eqref{eq:gdjP} in the second inequality, and \eqref{eq:JJ} in the third. The thesis follows by taking the inf on $x$ and $g$ in the LHS and the sup on $A$ in the RHS. 		 
	\end{proof}
	\subsection{FW and FDFW on sublevel sets of strongly convex smooth functions} \label{s:FWdirb}
	
	 In spite of the zig-zagging behaviour discussed in the introduction, in any fixed point of a polytope the FW directions do satisfy the sufficient slope condition, e.g. by Proposition \ref{eqphi} since they are selected by the AFW and the PFW when the active set is a singleton. This property can be extended (uniformly) to the boundaries of sublevel sets of smooth strongly convex functions. As a consequence, using that the FDFW selects FW directions on the boundary of strictly convex sets, we also obtain a global slope bound for this method. 
	\begin{proposition} \label{p:psmothsconv}
Let $h: \R^n \rightarrow \R$ be $\mu_h$-strongly convex and with $L_h$-Lipschitz gradient, $a > \min_{x \in \R^n} h(x)$, $\OM= \{x \in \R^n \ | \ h(x) \leq a\}$. Then:
	\begin{align}
			\SB_{\tx{FW}}(\OM, \partial \OM) \geq  \frac{u_h}{2L_h} \, , \label{eq:FWb} \\
		\SB_{\tx{FDFW}}(\OM) \geq  \frac{u_h}{2L_h} \, .  \label{eq:FDFWb}
	\end{align}
	\end{proposition}	
	\begin{proof}
 Since for every $x \in \partial \OM$ $$N_{\OM}(\x) = \tx{cone}(\{\nabla h (x)\}) \, ,$$ we can apply Lemma \ref{p:J} with $J(x) = \nabla h(x)$. We obtain
		\begin{equation} \label{eq:sublevel}
			\SB_{\tx{FW}}(\OM, \partial \OM) \geq \inf_{\substack{h(y) = h(\x) = a \\ \x \neq y}} \frac{\s{\nabla h(y)}{y-\x}}{\n{\nabla h(\x) - \nabla h(y)}\n{y-\x}} \, . 	
		\end{equation} 
		First, by $L_h$-smoothness we have: 
		\begin{equation}\label{eq:smooth}
			\n{\nabla h(\x) - \nabla h(y)}	\leq L_h \n{\x - y} \, .
		\end{equation}
		Second, by $\mu_h$ strong convexity
		\begin{equation} \label{eq:strc}
			\s{\nabla h(y)}{y-\x} \geq h(\x) - h(y) + \frac{\mu_h}{2}\n{y - \x}^2 = \frac{\mu_h}{2}\n{y - \x}^2 \, ,
		\end{equation}
		where we used $h(y) = h(\x) = a$ in the equality. The first part of the thesis follows applying \eqref{eq:smooth} and \eqref{eq:strc} to the RHS of \eqref{eq:sublevel}.  	\\
	By strong convexity for every $y \in \partial \OM$
	\begin{equation*}
	\frac{	\mu_h}{2}\n{y - x^*}^2 \leq h(y) - h^* = a - h^* \, ,
	\end{equation*} 
	and therefore $\n{y - x^*} \leq \sqrt{2(a - h^*)/\mu_h}$. This clearly implies 
	\begin{equation} \label{eq:Dbound}
	D \leq  2\sqrt{2(a - h^*)/\mu_h} \, .
	\end{equation}
	Analogously, by $L_h$ smoothness for every $y \in \partial \OM$ 
	\begin{equation} \label{eq:wbound}
	\frac{L_h}{2} \n{y - x^*}^2 \geq h(x) - h^* = a - h^* \, ,
	\end{equation}
	and therefore $\n{y - x^*} \geq \sqrt{2(a - h^*)/L_h}$. We can now give a slope bound for  $x \in \tx{int}(\OM)$:
	\begin{equation} \label{DSBFDstrict}
	\begin{aligned}
		&\tx{DSB}_{\tx{FD}}(\OM, x, g) \geq \frac{\s{g}{d^{\tx{FW}} + d^A }}{2D\pi_x(g)} = \frac{\s{\hat{g}}{d^{\tx{FW}} + d^A }}{2D} \geq \frac{1}{4}\sqrt{\frac{\mu_h}{2(a - h^*)}}\max_{s, q \in \OM} \s{\hat{g}}{s - q}  \\
	   \geq & \frac{1}{4}\sqrt{\frac{\mu_h}{2(a - h^*)}}2\sqrt{\frac{2(a - h^*)}{L_h}} = \frac{1}{2}\sqrt{\frac{\mu_h}{L_h}} \geq \frac{\mu_h}{2L_h} \, ,
	\end{aligned}	
	\end{equation}
	where we used \eqref{DSBineqFD} in the first inequality, $\pi_x(g) = \n{g}$ in the first equality, \eqref{eq:Dbound} in the second inequality, $x^* \pm \sqrt{2(a - h^*)/L_h}\hat{g} \in \OM$ by \eqref{eq:wbound} in the third inequality, and $\mu_h/L_h \leq 1$ in the last one. But on the boundary the FDFW selects FW directions so that \eqref{eq:FWb} holds also for the FDFW, and in particular \eqref{eq:FDFWb} follows by combining \eqref{eq:FWb} with \eqref{DSBFDstrict} for points on the interior. 
			
	\end{proof} 
	\begin{remark}
The slope bound \eqref{eq:FDFWb} clearly holds for any method selecting FW directions on the boundary and other directions on the interior satisfying the sufficient slope condition for $\mu_h/2L_h$ (e.g. the negative gradient). 
	\end{remark}

	\subsection{Orthographic retractions for convex sets with smooth boundary} \label{s:APD}
In this section we show an example of how the SSC can be employed effectively to deal with short steps also in non linearly constrained settings. More specifically, we consider here the case of $\OM$ compact, convex, full dimensional and with smooth boundary of class $C^1$ (e.g., the unitary ball w.r.t. $\n{\cdot}_p$, with $p \in (1, \infty))$. In this case for $x \in \partial \OM$ we can define $J(x)$ as the outward pointing normal and $T_x \OM$ as the tangent hyperplane to $\partial \OM$ in $x$. For $x \in \partial \OM$, $u \in T_x \OM$, we then define $\tip(x, u)$ as the first intersection between $\OM$ and $\{x+u\} + \tx{line}(-J(x))$. By smoothness $P(x, u)$ is always defined for $\n{u}$ small enough. This transformation is called an orthographic retraction, and can be easily extended to manifolds with arbitrary codimension. Orthographic retractions are widely used in manifold optimization for manifolds of class at least $C^2$ (see e.g. \cite{absil2012projection}, \cite{absil2015low}, \cite{kaneko2012empirical}, \cite{zhang2018robust}), in which case it can be proved that they are differentiable retractions \cite{absil2012projection}. In the codimension 1 case orthographic retractions were used in \cite{balashov2019gradient} for optimization over proximally smooth surfaces, and in \cite{levy2019projection} for optimization on sublevel sets of convex functions with Lipschitz gradient. In these works the additional regularity assumptions ensure that the manifold is sufficiently flat so that steps with length proportional to $\n{g_{\OM}(x)}$ are effective. Here we do not need additional regularity assumptions, thanks to the SSC which allows our method to choose shorter steps while skipping gradient computations.

		\begin{figure}[h]		
		\begin{subfigure}[t]{0.45\textwidth}
			\centering
			\includegraphics[width=0.8\textwidth]{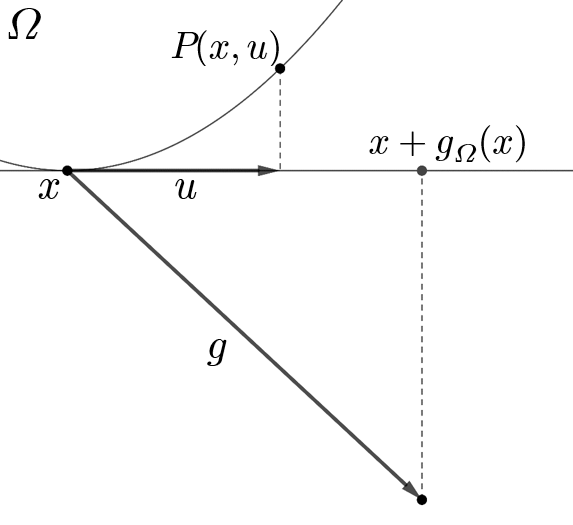} 
			\caption{Projection on $T_x\OM$ and orthographic projection \\ from $T_x\OM$.}
			\label{fig:subim3}
		\end{subfigure}	
  \hfill	
		\begin{subfigure}[t]{0.45\textwidth}
			\centering
			\includegraphics[width=0.8\textwidth]{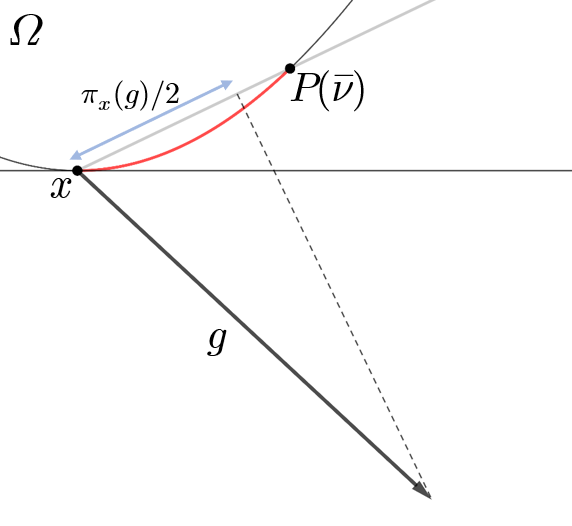}
			\caption{The directions in $\sop(x, g)$ point in the arc between $x$ and $P(\bn) = P(x, \bn(x, g) \ro(x))$ for $\T=\frac{1}{2}$.}
			\label{fig:subim4}
		\end{subfigure}
	
		\caption{}
		
	\end{figure}	
For $g \in \R^n$ we denote with $\rj(x)$ and $\ro(x)$ the components of $g$ along $J(x)$ and $T_x \OM$, so that $g = \rj(x) + \ro(x)$, with $\n{\ro(x)} = \pi_x(g)$. We can finally define the Short Orthographic Retractions (SOR) method directions $\sop(x, g)$ simply as $g$ if $g \in \tx{int}(T_{\OM}(x))$, and using orthographic projections when $(x, g)$ is in
$$ \tom = \left\{ (x, g) \in \partial \OM \times \R^n \ | \ g \notin \tx{int}(T_{\OM}(x)) \cup N_{\OM}(x) \right\} \, .$$
More precisely, in this case we define 
\begin{equation*} 
\sop(x, g) = \left\{P(x, \lambda \ro(x)) \ | \ \lambda \in (0, \bn(x, g)]  \right\} \, ,
\end{equation*}
where $\bn(x, g) > 0$ is a proper "shrinking coefficient", ensuring that the orthographic projection is taken close to $x$ enough to satisfy the usual slope condition for a fixed $\T \in (0, 1)$:
\begin{equation} \label{bn}
	\bn(x, g) = \sup \left\{ \lambda \in \R_{\geq 0} \ | \ \frac{\s{g}{P(x, \lambda \ro(x)) - x}}{\pi_x(g)\n{P(x, \lambda \ro(x)) - x}} \geq \T  \right\} \, .
\end{equation}
By smoothness we have $\bn(x, g) > 0$ whenever $(x, g) \in \tom$. It is also immediate to check from the definition that for $\bar{\lambda} > 0$ we have  $\bar{\lambda}\bn(x, \bar{\lambda}g) = \bn(x , g) = \bn(x, \hat{g})/\n{g} $. We now prove the anticipated slope lower bound:
\begin{proposition} \label{sopOM}
	In the setting introduced above $\SB_{\sop}(\OM) \geq \T$.
\end{proposition}
\begin{proof}
	When $g\in \tx{int}(T_{\OM}(x))$ we obviously have $\DSB_{\sop}(\OM, x, g) = 1 > \T$ since $\sop(x, g) = \{g\}$. By the definition of $\bn$ if $(x, g) \in \tom$ then $P(x, \bn(x, g)\ro(\x)) \in \partial \OM$	and (see also figure \ref{fig:subim4}) $$v = P(x, \bn(x, g)\ro(x)) - x \in \tx{int}(T_{\OM}(x)) \, .$$	
	In particular $P(x, \lambda \bn(x, g) \ro(x))$ is in $\tx{conv}(\{x + \lambda \bn(x, g) \ro(x), x + \lambda v\})$ strictly below $x + \lambda v$ for every $\lambda \in (0, 1)$. Therefore, for 
	$v_{\lambda} = P(x, \lambda \bn(x, g) \ro(x)) - x $
	\begin{equation} \label{strictbT}
	\frac{\s{g}{v_{\lambda}}}{\pi_x(g)\n{v_{\lambda}}} > \frac{\s{g}{v}}{\pi_x(g)\n{v}} \geq \T \, ,
	\end{equation}
	where the second inequality follows by the inequality in the definition \eqref{bn} of $\bn(x, g)$. The equation above directly implies the desired bound
	\begin{equation*}
	\begin{aligned}
	\DSB_{\sop}(\OM, x, g) = & \inf_{\lambda \in (0, 1] } \frac{\s{g}{v_{\lambda}}}{\pi_x(g)\n{v_{\lambda}}}  \\ 
	= & \frac{\s{g}{v}}{\pi_x(g)\n{v}} \geq \T \, . 
	\end{aligned}
	\end{equation*}
    and taking into account the previous case we can conclude  $\SB_{\sop}(\OM)\geq \T$. 
	
\end{proof} 
By definition of $\sop(x, g)$ we only need a lower bound on $\bn(x, u)$ to compute a direction $d \in \sop(x, g)$. While in practice such bounds can be obtained with an adaptive strategy, for sublevel sets simple lower bounds depending from a smoothness modulus are sometimes available (see Proposition \ref{proximally_smooth_ex} and Remark \ref{proximally_smooth_rem} in the appendix for some examples). The general principle in this case is that for every $x \in \partial \OM$ a "simple" smooth full dimensional convex set $\bar{B}(x) \ni x$ (e.g. a ball) is contained in $\OM$, with the same tangent hyperplane in $x$. We can therefore bound $\bn$ with the shrinking coefficient relative to $\bar{B}(x)$, which is much easier to control.

The $\sop$ method directions are well suited for the SSC, given that they satisfy the sufficient slope condition, but in general they are too short to ensure descent sequence like conditions in a single step. Still, even with the SSC some lower bound on the norm of the directions chosen is needed to obtain finite termination results. In other words, it is necessary to "throw away" the shortest directions selected by the $\sop$. Here for simplicity we consider the method
\begin{equation*}
	\bsop(\OM, x, g) = \left\{P(x, \lambda \ro(x)) \ | \  \min  ( \hn, \hn \bn(x, g) )   \leq \lambda \leq  \bn(x, g)  \right\} \, ,
\end{equation*}
where $\hn \in (0, 1]$ is a fixed constant in the rest of the paper. It is clear that 
$$\SB_{\bsop}(\OM) \geq \SB_{\sop}(\OM) \geq \T \, , $$
given that the directions selected by the $\bsop$ are a subset of the ones selected by the $\sop$. In the next lemma we prove that also for the $\bsop$ we have convergence for linear objectives, and in particular finite SSC termination thanks to Lemma \ref{finiteTermCrit}.
\begin{lemma} \label{prightarrow0}
If the method $\bsop$ applied to the linear function $L_g(x) = -\s{g}{x}$ generates a sequence $\{y_j\}$ with stepsizes maximal w.r.t. boundary conditions (as in \eqref{eq:alphamax}), then 
	\begin{equation} \label{zj:fin}
		\lim \pi_{y_j}(g) = 0 \, .
	\end{equation}
	In particular, the SSC with the $\bsop$ always terminates in a finite number of steps.
\end{lemma}
\begin{proof}
  We just need to prove \eqref{zj:fin} since thanks to Lemma \ref{finiteTermCrit} finite SSC termination immediately follows. Furthermore, if the sequence is finite, terminating in a stationary point for $-g$, \eqref{zj:fin} is obvious, so that from now on we always assume that $\{y_j\}$ is an infinite sequence.
	Let $$I = \{ j \in \mathbb{N}_0 \ | \ (y_j, g) \in \tom \}\, ,$$ and let $\{\lambda_j\}_{j \in I}$  be such that $	d_j = P(y_j, \lambda_j \ro(y_j)) - y_j$.
 We first claim that if the sequence doesn't stop in a stationary point then $\mathbb{N} \subset I$. To start with, since the stepsizes must always be maximal, $y_j \in \partial \OM $ for every $j \geq 1$. Now if $g \in \tx{int}(T_{\OM}(y_j))$ and also $g \in \tx{int}(T_{\OM}(y_{j+1})) = \tx{int}(T_{\OM}(y_{j} + \alpha_j g)) $ then necessarily $\tx{int}(\OM) \ni y_{j} + \alpha_j g = y_{j + 1}$, a contradiction. On the other hand if $(y_j, g) \in \tom$ then $y_j + d_j \in \partial \OM$ with $y_j + d_j - \lambda J(x) \notin \OM$ for every $\lambda > 0$. Then by convexity it is easy to see 
 $$ y_{j} + \alpha_j d_j + \lambda_1 \ro(\OM) + \lambda_2 J(x) \notin \OM  $$
 for every $\lambda_2 \geq 0, \lambda_1 > 0$,  and thus in particular 
 $$ y_{j + 1} + \lambda g = y_{j} + \alpha_j d_j + \lambda \ro(y_j) + \lambda \n{\rj(y_j)}J(x) \notin \OM $$
 for every $\lambda > 0$. Equivalently, $g \notin \tx{int}(T_{\OM}(y_{j + 1}))$ as desired. \\
Let $\{y_{j(k)}\} \rightarrow \bar{y}$ be a convergent subsequence of $\{y_j\}$, and assume by contradiction
$$	\lim_{k \rightarrow \infty} \pi_{y_{j(k)}}(g) = \pi_{\bar{y}}(g) \neq 0 \, ,  $$
so that $(\bar{y}, g) \in \tom$ and $\bn(\bar{y}, g) > 0$.
The sequence $f_j = \s{g}{y_j}$ is strictly decreasing with limit $\bar{f} \in \R$ by compactness. In particular we have 
\begin{equation} \label{bounding}
\begin{aligned}
& \s{g}{y_1} - \bar{f} = \sum_{j = 1}^{\infty} \s{g}{y_j - y_{j + 1}} \geq \sum_{k=1}^{\infty} \s{g}{y_{j(k)} - y_{j(k + 1)} } \\ 
\geq &  \sum_{k=1}^{\infty} \s{g}{P(y_{j(k)}, \lambda_{j(k)} \ro(y_{j(k)})) - y_{j(k)}} \geq \T \sum_{k=1}^{\infty} \n{P(y_{j(k)}, \lambda_{j(k)} \ro(y_{j(k)})) - y_{j(k)}} \pi_{y_{j(k)}}(g) \\
\geq & \T \sum_{k=1}^{\infty} \lambda_{j(k)} \pi_{y_{j(k)}}(g)^2 \geq  \T \hn \sum_{k=1}^{\infty}  \pi_{y_{j(k)}}(g)^2 \min (  \bn(y_{j(k)}, g), 1) \, ,
\end{aligned}
\end{equation} 
where we used $\alpha_{j(k)} = \alpha_{j(k)}^{\max} \geq 1$ in the second inequality, $\SB_{\bsop}(\OM) \geq \T$ in the third one, $$\n{P(y_{j(k)}, \lambda_j \ro(y_{j(k)})) - y_{j(k)}} \geq \lambda_j \pi_{y_{j(k)}}(g) $$ in the fourth one, and $\lambda_j \geq \min ( \hn \bn(y_{j(k)}, g), \hn)$ in the last one. By construction $\pi_{y_{j(k)}}(g)$ converges to $\pi_{\bar{y}}(g) \neq 0$, and therefore \eqref{bounding} implies $\bn(y_{j(k)}, g) \rightarrow 0$, since the last sum is bounded by the LHS. But by the lower semicontinuity property proved in Lemma \ref{bncontinuity} this is in contradiction with $\bn(\bar{y}, g) > 0$.  
\end{proof}
\begin{remark} \label{generalizing}
	If instead of $P$ we have a retraction $R(x, u)$, defined in a neighborhood $U$ of $\partial \OM \times \{0\} \subset T \partial \OM$, such that
\begin{equation} \label{retraction}
	R(x, tu) = x + tu + o(t) \, ,
\end{equation}
then we can always define a shrinking coefficient $\bn_R$ analogous to $\bn$ and a method $\mathcal{A}_R$ analogous to the $\sop$. The essential property needed to have finite termination is that for $x_j \rightarrow \bar{x} $ such that $(\bar{x}, g) \in \tom$ and $d_j \in \A_R(x_j, g)$ 
\begin{equation*}
 \liminf_{j \rightarrow \infty} \n{d_j} \neq 0 \, .
\end{equation*}
This is satisfied by the $\bsop$ thanks to Lemma \ref{bncontinuity} and can be satisfied by analogous variants e.g. if the remainder $o(t)$ in \eqref{retraction} is uniform for $x, u$ varying in $U$.  
\end{remark}

	\subsection{Directions on product domains}
	Assume that the feasible set is block separable, that is $\OM = \OM_{(1)} \times ... \times \OM_{(\m)}$ with $\OM_{(i)} \subset \R^{n_i}$ compact and convex for $1 \leq i \leq m$. The next proposition states that we can select directions with positive slope bound for $\OM$ by properly weighting directions selected by algorithms with positive slope bound in each of the factors. For a block vector $x \in \R^n = \R^{n_1}\times ... \times \R^{n_m}$ we denote with $x_{(i)} \in \R^{n_i}$ the component corresponding to the $i - $th block, so that $x = (x_{(1)}, ..., x_{(\m)})$.  
	\begin{proposition}  \label{prop:domProd}
		For $\OM$ defined as above and $1\leq i \leq \m$.  let $\A^i$ be such that $\SB_{\A^i}(\OM_{(i)}) = \tau_i > 0$, and let $\tau^* = \min_{1 \leq i \leq \m} \tau_i$.
		\begin{enumerate}			
			\item Assume that $d \in \A(x, g)$ if and only if for every $i \in [1:m]$ we have
			\begin{equation} \label{def:axrbigo}
			d_{(i)} = w_i \hat{c}_i, \ \tx{with } c_i \in \A^i(x_{(i)}, g_{(i)}), \  w_i = \s{g}{\hat{c}_i}  \, .
			\end{equation}
			Then $\SB_{\A}(\OM) = \tau^* $.
			\item Assume that $d \in \A(x, g)$ if and only if for every $i \in [1:m]$ we have
			\begin{equation}\label{def:blockdir}
          	d_{(i)} = w_i \hat{c}_i, \ \tx{with } c_i \in \A^i(x_{(i)}, g_{(i)}), \ w_i \geq 0
			\end{equation}
		 for a weight vector $w \in \R^m \sm \{0\}$ such that
			\begin{equation*}
				\argmax_{i \in [1:m]} w_i \cap \argmax_{i \in [1:m]} \s{g}{\hat{c}_i} \neq \emptyset \, .
			\end{equation*}
		\end{enumerate} 
Then $\SB_{\A}(\OM) \geq \frac{\tau^*}{m}$. 
	\end{proposition}	
	
	\begin{proof}
	
		For $\OM$ product of closed convex sets the normal cone is the product of the normal cones of the factors (see e.g. \cite[Table 4.3]{Aubin_2009}) 
		\begin{equation} \label{product}
			N_{\OM}(x) = N_{\OM_{(1)}}(x_{(1)}) \times ... \times N_{\OM_{(m)}}(x_{(m)}) \, .
		\end{equation}
		As a consequence of \eqref{product} for $g \in \R^n$
		\begin{equation}\label{NOMsum}
			\dist(g, N_{\OM}(x)) = \sqrt{\sum_{i = 1}^{\m} \dist(g_{(i)}, N_{\OM_{(i)}}(x_{(i)}))^2} \, .
		\end{equation}
	For $d \in \A(x, g)$ we have
		\begin{equation} \label{eq:2p}
			\n{d} = \sqrt{\sum_{i = 1}^{\m} \n{d_{(i)}}^2} = \sqrt{\sum_{i = 1}^{\m} w_i^2} 
		\end{equation}
		and
		\begin{equation} \label{eq:3p}
			\s{g}{d} = \sum_{i = 1}^{\m} \s{g_{(i)}}{d_{(i)}} = \sum_{i = 1}^{\m} w_i \s{g_{(i)}}{\hat{c}_i} \, ,
		\end{equation}
		with $c_i \in \A^i(x_{(i)}, g_{(i)})$ for $1 \leq i \leq m$.  \\
			Let $\pi_x^i(g) = \n{\pi(T_{\OM_{(i)}}(x_{(i)}), g_{(i)})}$. Then we get
		\begin{equation} \label{eq:1p}
		\pi_x(g) = \sqrt{\sum_{i = 1}^{\m} \pi_x^i(g)^2} \leq \frac{1}{\tau^*}\sqrt{\sum_{i = 1}^{\m} \s{g_{(i)}}{\hat{c}_i}^2}
		\end{equation}
		by applying Proposition \ref{NWessential} to the RHS of \eqref{NOMsum} in the equality, and using 
	 $$ \s{g_{(i)}}{\hat{c}_i} \geq \DSB_{\A^i}(\OM_{(i)}, x_{(i)}, g_{(i)})  \pi_x^i(g) \geq \SB_{\A^i}(\OM_{(i)})  \pi_x^i(g) = \tau_i  \pi_x^i(g) \geq \tau^* \pi_x^i(g)$$ 
	 in the inequality. \\
	 1. For $d \in \A(x, g)$ as in \eqref{def:axrbigo} we have
		\begin{equation*} 
 \frac{\s{g}{d}}{\pi_x(g)\n{d}} = \frac{\sum_{i = 1}^{\m} \s{g_{(i)}}{\hat{c}_i}^2}{\sqrt{\sum_{i = 1}^{\m} \pi_x^i(g)^2}\sqrt{\sum_{i = 1}^{\m} \s{g_{(i)}}{\hat{c}_i}^2}} = \frac{\sqrt{\sum_{i = 1}^{\m} \s{g_{(i)}}{\hat{c}_i}^2}}{\sqrt{\sum_{i = 1}^{\m} \pi_x^i(g)^2}} \geq \tau^* \, ,
		\end{equation*}
		where we used \eqref{eq:1p}, \eqref{eq:2p} and \eqref{eq:3p} in the first equality, and the second part of \eqref{eq:1p} in the inequality. Then 
		\begin{equation*}
		\DSB_{\A}(\OM, x, g) = \inf_{d \in \A(x, d)}	 \frac{\s{g}{d}}{\pi_x(g)\n{d}} \geq \tau^* \, ,
		\end{equation*} 
		and the lower bound $\SB_{\A}(\OM) \geq \tau^*$ follows by taking the inf on $x, g$. \\
		Let $\bar{m}$ such that $\tau_{\bar{m}} = \tau^*$. It is immediate to obtain the opposite inequality $ \SB_{\A}(\OM) \leq \tau^* $ by considering the block gradients of the form $g = (0, ..., g_{(\bar{m})}, ..., 0)$ with all zeros outside the block $\bar{m}$, where for every such $g$ we have 
		\begin{equation*}
			\DSB_{\A}(\OM, x, g) = \DSB_{\A^{\bar{m}}}(\OM_{(\bar{m})}, x_{(\bar{m})}, g_{(\bar{m})}) \, ,
		\end{equation*}
		by the definition of $\A$. \\
		2. Let $l \in \argmax_{i \in [1:m]} w_i \cap \argmax_{i \in [1:m]} \s{g}{\hat{c}_i}$. Then for $d \in \A(x, g)$ as in \eqref{def:axrbigo}
			\begin{equation*} 
		\frac{\s{g}{d}}{\pi_x(g)\n{d}} = \frac{\sum_{i = 1}^{\m} w_i \s{g_{(i)}}{\hat{c}_i}}{\sqrt{\sum_{i = 1}^{\m} \pi_x^i(g)^2}\sqrt{\sum_{i = 1}^{\m} w_i^2}} \geq \frac{1}{\sqrt{m}} \frac{\s{g_{(l)}}{\hat{c}_l}}{\sqrt{\sum_{i = 1}^{\m} \pi_x^i(g)^2}} \geq \frac{\tau^*}{\sqrt{m}} \frac{\s{g_{(l)}}{\hat{c}_l}}{\sqrt{\sum_{i = 1}^{\m} \s{g_{(i)}}{\hat{c}_i}^2}} \geq \frac{\tau^*}{m} \, ,
		\end{equation*}
		where we used \eqref{eq:2p} and \eqref{eq:3p} in the first equality, $l \in \argmax_{i \in [1:m]} w_i$ in the first inequality, \eqref{eq:1p} in the second inequality, and $l \in \argmax_{i \in [1:m]} \s{g}{\hat{c}_i} $ in the last one. 

	\end{proof}	

	The directions given in \eqref{def:axrbigo} have components normalized and then multiplied by the slopes in the corresponding blocks. Thus in particular $\A$ does longer steps in the blocks with greater slope, which seems a rather intuitive way to obtain direction satisfying the sufficient slope condition. As an example, consider the case of a product of simplices
	$$ \OM = \Delta_{n_1 - 1} \times ... \times \Delta_{n_m - 1} \, ,  $$ 
	relevant e.g. for structural SVMs optimization \cite{lacoste2013block} and suitable for projection-free methods. Given $x \in \OM$, $g \in \R^n$ it can be readily seen that the direction computed by the standard PFW is of the form $d_{(i)} = e_{s(i)} - e_{q(i)}$  for $i \in [1:m]$, with  $s(i) \in \argmax_{j \in [1:n_i]} \s{g_{(i)}}{e_j}$ and $q(i)$ depending on the active set. Meanwhile, the direction suggested by Proposition \ref{def:axrbigo} when applying the PFW to every block is of the form $d_{(i)} = w_i \hat{d}^{\tx{PFW}_i}$ for $i \in [1:m]$, with $w_i = \s{g_{(i)}}{\hat{d}^{\tx{PFW}_i}}$ and $d^{\tx{PFW}_i} \in \tx{PFW}(x_{(i)}, g_{(i)})$.	
	
	We now state a termination result for the SSC procedure on product domains. Essentially, while finite termination holds for a reasonable choice of the maximal stepsize even for products of blocks with the AFW or the PFW, the blocks with the FDFW or the $\bsop$ must be dealt with one at a time. The result in particular implies finite termination of the SSC for the $\bsop$. 
	\begin{proposition} \label{prop:SSCprodf}
	Let $\A$ be as in \eqref{def:blockdir}, with $\A^i$ the AFW, the PFW, the FDFW or the $\bsop$ for $i \in [1:m]$, and let $I \subset [1:m]$ the set of blocks with the FDFW or the $\bsop$. Assume that during the SSC only linear minimizers are added to the active sets in blocks with the AFW or the PFW. If 
	\begin{itemize}	
	\item 	for every maximal step of $\A$ in the SSC a maximal step is done in $\A^i$ for some $i \in [1:m]$,
	\item whenever a maximal step is done by $\A^l$ with $l \in I$ then the directions in the other blocks in $I$ are $0$ and 
	\begin{equation*} 
		l \in \argmax_{i \in [1:m]} \s{\hat{c}_i}{g_{(i)}} \, ,
	\end{equation*}
	\end{itemize}
  then the SSC terminates.
	\end{proposition}	
	\begin{proof}
		Assume by contradiction that the SSC generates an infinite sequence $\{y_j\}$, leading to an infinite sequence of maximal steps. 
		For every $i \in I^c$ by Proposition \ref{AFWPFWfiniteTerm} the method $\A^{(i)}$ can do limited number of maximal steps (consecutive or not). Then for $j \geq \bar{l}$ large enough every maximal step must take place in a block in $I$. By assumption then the blocks in $I$ must be changed one at a time and with maximal steps, so that each $\A^i$ for $i \in I$ does consecutive maximal steps in its block. In particular, for some $l \in I$ infinite consecutive maximal steps are done in the $l-$th block. Since the FDFW can do a finite number of consecutive maximal steps, necessarily $\A^l$ must be the $\bsop$. Let now $y_{ji} = (y_{j})_{(i)}$, $d_{ji} = (d_j)_{(i)}$ and $p_{ji} = \pi_{y_{ji}}(g_{(i)})$.	Let $j(k)$ be the subsequence of indexes corresponding to maximal steps in the $l-$th block. Then $p_{j(k)l} \rightarrow 0$ by Lemma \ref{prightarrow0}. Whence $p_{j(k)} \rightarrow 0$, since
     \begin{equation*}
    p_{j(k)}  \leq \frac{1}{\tau^*} \sqrt{\sum_{i=1}^{m}\s{\hat{d}_{j(k)i}}{g_{(i)}}^2} \leq \frac{\sqrt{m}}{\tau^*} \s{\hat{d}_{j(k)l}}{g_{(l)}} \leq  \frac{\sqrt{m}}{\tau^*}  	p_{j(k)l} \, ,
     \end{equation*}
     where we used \eqref{eq:1p} in the first inequality, $l \in \argmax_{i \in [1:m]} \s{\hat{d}_{j(k)i}}{g_{(i)}}$ in the second, and Proposition \ref{NWessential} in the third one. We therefore have all the hypotheses to apply Lemma \ref{finiteTermCrit}.

	\end{proof}

	We notice that the assumption on the maximal steps can be satisfied  by simply taking $\alpha_{\max}$ as the minimum among the maximal stepsizes in the blocks. As for the second assumption, it is always compatible with \eqref{def:blockdir}, where e.g. both are satisfied by directions with only one weight different from $0$, located in a block with maximal slope.

	\subsection{Linear convergence for quadratic objectives on polytopes with FW variants}\label{s:QP}
	In this section, we first prove that a uniform KL property together with the usual slope condition on $\A$ imply an asymptotic linear convergence rate for a sequence generated by Algorithm \ref{tab:3}. This result together with a KL property of quadratic programming problems proved in \cite{forti2006} and then more in general in \cite{li2018calculus} imply, in turn, asymptotic linear convergence rates for several FW variants.  
	
	\begin{proposition} \label{linear convergence}
		Let us consider the set $\mathcal{X}$ of stationary points of Problem \eqref{eq:mainpb}.	
		Assume that every point $x^*\in \mathcal{X}$  satisfies the KL inequality in $B_{\delta(x^*)}(x^*) \cap [f(x^*) < f < \eta(x^*)]$ for \mbox{$\varphi_{x^*}(t) = 2M_{x^*}t^{\frac{1}{2}}$}and suitable $\delta(x^*) > 0$, $\eta(x^*) > f(x^*)$.
		
		Let $\{z_k\}$ be a sequence generated by Algorithm \ref{tab:3},  with the method $\mathcal{A}$  in the SSC  satisfying the following:
		\begin{itemize}
			\item  the SSC  procedure always terminates in a finite number of steps;
			\item $\SB_{\mathcal{A}}(\OM)=\tau>0$.
		\end{itemize}  
		Then $z_k \rightarrow \tilde{x}^* \in \mathcal{X}$ at an asymptotic rate
		\begin{equation*}
			\n{z_k - \tilde{x}^*} = O\left( \left(1 + \frac{a}{b^2M^2} \right)^{-\frac{k}{2}} \right) \, ,
		\end{equation*}		
		where $a = L/2$, \mbox{$b=\tau/(L(1 + \tau))$}, and $M$ is a constant depending on $\OM, f,$ but not $z_0$. \\
	\end{proposition}	
	\begin{proof}
		By continuity for every $x^* \in \mathcal{X}$ we can define $\hat{\delta}(x^*) \in (0, \delta(x^*))$ so that conditions \eqref{n4fx0fx*} and \eqref{corcond} hold for every $ x_0 \in B_{\hat{\delta}(x^*)}(x^*) \cap [f(x^*) < f]$. By compactness we can then take a finite subset $\mathcal{X}^*$ of stationary points such that
		$$ \bigcup_{\tilde{x} \in \mathcal{X}^*} B_{\hat{\delta}(\tilde{x})}(\tilde{x}) \supset  \mathcal{X} \, . $$
		Furthermore, by Corollary \ref{cor:stationary} the sequence $\{f(z_{k})\}$ is decreasing and converges to $f^* \in \R$, and $\{z_k\}$ converges to $\mathcal{X}$. Then for some $\bar{k}$ large enough and some $x^* \in \mathcal{X}^*$ such that $f(x^*) = f^*$ we must have $z_{\bar{k}} \in B_{\hat{\delta}(x^*)}(x^*) $  . Given that $f(z_{k})$ is decreasing and converges to $\tilde{f}^*$, also the assumption \eqref{cond:boundc} is satisfied for $\{x_k\} = \{z_{\bar{k} + k}\}$. By Corollary \ref{cor:corA2conv} applied to $\{x_{k}\} = \{z_{\bar{k} + k}\}$ we obtain 
		$$\{x_k\} \rightarrow \tilde{x}^* \in B_{\hat\delta(x^*)}(x^*) \cap \mathcal{X} \, ,$$ 
		with
		\begin{equation} \label{xkconv}
			\n{x_k - \tilde{x}^*} = O\left( \left(1 + \frac{a}{b^2M_{x^*}^2} \right)^{-\frac{k}{2}} \right) \, 	
		\end{equation}
		by Corollary \ref{holderian},	where the implicit constant can be taken independent from $z_{\bar{k}}$.
		To conclude,
		\begin{equation*}
			\n{z_{k} - \tilde{x}^*}= O\left( \left(1 + \frac{a}{b^2M_{x^*}^2} \right)^{-\frac{k - \bar{k}}{2}} \right)
			=  O\left( \left(1 + \frac{a}{b^2M_{x^*}^2} \right)^{-\frac{k}{2}} \right) = O\left( \left(1 + \frac{a}{b^2M^2} \right)^{-\frac{k}{2}} \right)
		\end{equation*}
		for $M = \max\{ M_{\tilde{x}} \ | \ \tilde{x} \in \mathcal{X}^* \}$ independent from $z_0$, where we used \eqref{xkconv} in the first equality.	
	\end{proof}

	We can now easily prove the result for quadratic programming problems with compact domain. 
	\begin{corollary}
		The conclusion of Proposition \ref{linear convergence}  holds in particular if $f(x) = \frac{1}{2}x^{\intercal} Q x - b^{\intercal} x$, $\OM$ is a polytope, and the method $\mathcal{A}$ used in the SSC is the AFW, the PFW or the FDFW, with no points beside linear minimizers added in the active set during the SSC for the AFW and the PFW.
	\end{corollary}
	\begin{proof}
		By \cite[Corollary 5.2]{li2018calculus} the KL property needed is satisfied in particular when $f$ is quadratic and $\OM$ is a polytope. Moreover, when $\A$ is the AFW, the PFW or the FDFW and $\OM = \tx{conv}(A)$ with $|A| < + \infty$ is a polytope then $\SB_{\A}(\OM) > 0$ by Proposition \ref{eqphi}. In this setting we have also the finite termination property by Proposition \ref{AFWPFWfiniteTerm} and Proposition \ref{FDFWfiniteterm}. Therefore we have all the hypotheses to apply Proposition \ref{linear convergence}.  
	\end{proof}
	\begin{remark}
		If for every $x \in \mathcal{X}$ there exists $\delta$ such that $f(y) = f(x)$ for every $y \in \mathcal{X} \cap B_{\delta}(x)$ then  by \cite[Theorem 4.1]{li2018calculus} the condition of  Proposition \ref{linear convergence} is implied by the Luo Tseng error bound. When instead the condition does not hold, i.e. we do not have the same desingularizing function for every stationary point (even up to a constant), we can of course still give local results. As an example, if $\A$ is the FDFW or the $\bsop$, then by applying Algorithm \ref{tab:3} to the trust region subproblem and using the KL property proved in \cite{jiang2019h} we can obtain local converge rates around global minimizers. Specifically, we obtain the $O(1/k^2)$ rate given in Corollary \ref{holderian} for $\theta = 1/4$ in the (rather technical) critical case condition \cite[equation (24)]{jiang2019h}, and the linear rate given in Corollary \ref{holderian} for $\theta = 1/2$ otherwise. 
	\end{remark}

	\section{Conclusions} 	
	Projection-free first-order methods rely on the choice of good feasible descent directions, for which there needs to be a trade-off between slope and maximal stepsize. To adress this issue we proposed the SSC procedure, which allowed us to give a unified analysis for non convex objectives of several methods satisfying a sufficient slope condition. With this analysis we proved convergence rates independent from the number of short  steps, and showed how our SSC procedure easily adapts to different choices of descent directions. Finally, our work is the first to analyze projection-free methods on the class of smooth non convex objectives satisfying KL inequalities. 
	
	Future research directions include generalizing our framework to constrained stochastic optimization, as well as applications for the solution of real-world data science problems.
	\section{Appendix}	
	\subsection{Proofs for some elementary properties}\label{PAp}
	We need the following result to prove Proposition \ref{NWessential}:
	\begin{proposition} \label{ddual}
		Let $C$ be a closed convex cone. For every $y \in \mathbb{R}^n$ 
		\begin{equation*}
			\textnormal{dist}(C^*, y) = \sup_{c \in C} \s{\hat{c}}{y} \, .
		\end{equation*}
	\end{proposition}
	As stated in \cite{burke1988identification} this is an immediate consequence of the Moreau-Yosida decomposition: 
	$$ y = \pi(C, y) + \pi(C^*, y) \, . $$ 
	\begin{proof}[Proposition \ref{NWessential}]
		We first prove that 
		\begin{equation} \label{omtomeq}
			\sup_{h\in\Omega/\{\bar{x}\}} \left (g, \frac{h-\bar{x}}{\|h - \bar{x}\|}\right ) = \sup_{h\in T_{\Omega}(\bar{x}) \sm \{0\}} (g, \hat{h}) \, .
		\end{equation}
		Let $h \in T_{\Omega}(\bar{x}) \setminus \{0\}$. Then there exists sequences $\{\lambda_i\}$ and $\{h_i\}$ in $\R_{> 0}$ and $\Omega$ respectively such that $\lambda_i (h_i -\bar{x}) \rightarrow h$. In particular $\|\lambda_i (h_i - \bar{x})\| \rightarrow \|h\|$ so that we also have $\lambda_i (h_i -\bar{x})/\|\lambda_i(h_i - \bar{x})\| = (h_i - \bar{x})/\|h_i-\bar{x}\| \rightarrow \hat{h} $. Hence
		\begin{equation*}
			\textnormal{cl}\left (\left \{\frac{h - \bar{x}}{\|h - \bar{x}\|} \ | \ h \in \Omega \setminus \{0\} \right \}\right ) = \{\hat{h} \ | \ h \in T_{\Omega}(\bar{x})/\{0\}\} \, ,
		\end{equation*}
		and \eqref{omtomeq} follows immediately by the continuity of $(g, \cdot)$. \\
		Since $N_{\Omega}(\bar{x}) = T_{\Omega}(\bar{x})^*$ the first equality is exactly the one of Proposition \ref{ddual} if $g \notin N_{\Omega}(\bar{x})$, and it is trivial since both terms are clearly $0$ if $g \in N_{\Omega}(\bar{x})$.\\
		It remains to prove 
		\begin{equation*}
			\dist(N_{\Omega}(\bar{x}), g)= \|\pi(T_{\Omega}(\bar{x}), g)\| \, ,	
		\end{equation*}
		which is true by the Moreau - Yosida decomposition. 
	\end{proof}

In the next lemma we lower bound $\SB_{\tx{FW}}(\OM, \partial \OM)$ in terms of a relatively more manageable expression in the strictly convex smooth case. 
	\begin{lemma} \label{p:J}
		Let $\OM$ be a strictly convex smooth subset of $\R^n$, and let $J:\partial \OM \rightarrow \R^n$ be such that $J(\x) \in N_{\OM}(\x) \sm \{0\}$ for every $\x \in \partial \OM$ (so that in particular $N_{\OM}(\x) = \tx{cone}(J(\x))$). Then
		\begin{equation*} 
			\SB_{\tx{FW}}(\OM, \partial \OM) \geq \inf_{\substack{y, \x \in \partial \OM \\ \x \neq y}} \frac{\s{J(y)}{y-\x}}{\n{J(\x) - J(y)}\n{y-\x}} \, .
		\end{equation*} 
	\end{lemma}
	\begin{proof}
		Given $\bar{x} \in \OM$, $g \in \R^n \sm N_{\OM}(\bar{x})$, by strict convexity $g = \gamma J(y)$ for some $\gamma > 0$ and 
		$$\{y\} = \argmax_{x \in \OM} \s{g}{x} \, .$$ 
		In particular, $\tx{FW}(\bar{x}, g) = y - x$. Therefore 
		\begin{equation*}
			\DSB_{\tx{FW}}(\OM, \bar{x}, g) = \frac{\s{g}{y - \bar{x}}}{\pi_{\bar{x}}(g)\n{y - \bar{x}}} = \frac{\s{J(y)}{y - \bar{x}}}{\pi_{\bar{x}}(J(y))\n{y - \bar{x}}} \geq  \frac{\s{J(y)}{y - \bar{x}}}{\n{J(\bar{x}) - J(y)}\n{y - \bar{x}}}  \, ,
		\end{equation*}
		where we used $g = kJ(y)$ in the second equality and $N_{\OM}(\bar{x}) = \tx{cone}(J(\bar{x}))$ in the inequality. The lemma follows taking the inf in $g$ and $\bar{x}$, where for $g$ varying in $\R^n \sm N_{\OM}(\bar{x})$ the point $y$ varies in all $\partial \OM \sm \bar{x}$ by strict convexity. 
	\end{proof}

	We now prove a lower semicontinuity property for the shrinking coefficient $\bn$ introduced for orthographic projections on convex sets with smooth boundary in Section \ref{s:APD}. 
	\begin{lemma} \label{bncontinuity}
		Let $\OM$ be a full dimensional convex set with $C^1$ regular boundary, $\{x_j\} \subset \partial \OM$ be convergent to $x$. Let $g \in \R^n$ such that $\{(x_j, g)\} \subset \tom$ and $(x, g) \in \tom$. Then 
		$$\liminf \bn(x_j, g) \geq \bn(x, g) \, . $$
	\end{lemma}
\begin{proof}
Fix $\lambda \in (0, 1)$, and let $v = P(x, \bn(x, g) \ro(x)) - x$, $v_{\lambda} =  P(x, \lambda \bn(x, g)\ro(x)) - x $. Recall from the proof of Proposition \ref{sopOM} that $x + v_{\lambda}$ is in $\tx{conv}(x + \lambda \bn(x, g)\ro(x), x + \lambda v)$ minus $x + \lambda v \in \tx{int}(\OM)$. Therefore, for some $\varepsilon_{\lambda} > 0$
\begin{equation} \label{interiorintersection}
 x + v_{\lambda} + \gamma J(x) \in
	\begin{cases}
	\tx{int}(\OM) &\tx{ if }\gamma \in (0, \varepsilon_{\lambda}) \, ,\\
	\OM^c &\tx{ if }\gamma < 0 \, .
	\end{cases}	
\end{equation}
Since $J(x_j) \rightarrow J(x)$, $\ro(x_j) \rightarrow \ro(x)$ by smoothness, thanks to \eqref{interiorintersection} we have $$P(x_j, \lambda\bn(x, g) \ro(x_j)) \rightarrow P(x, \lambda\bn(x, g) \ro(x)) =x + v_{\lambda} \, .$$ Thus 
\begin{equation*}
	\frac{\s{P(x_j, \lambda \nu(x, g)\ro(x_j)) - x_j}{g}}{\pi_{x_j}(g) \n{P(x_j, \lambda \nu(x, g)\ro(x_j)) - x_j}} \rightarrow 	\frac{\s{v_{\lambda}}{g}}{\pi_{x}(g) \n{v_{\lambda}}} \, ,
\end{equation*}
and for $j$ large enough by \eqref{strictbT} the LHS must be greater than $\T$. In particular, $\bn(x_j, \ro(x_j)) \geq \lambda\nu(x, g)  $. The thesis follows since $\lambda \in (0, 1)$ is arbitrary. 
\end{proof}
We conclude this section by giving a practical example of lower bound for the shrinking coefficient $\bn$ introduced in Section \ref{s:APD}.
	\begin{proposition} \label{proximally_smooth_ex}
	Let $h$ be a uniformly smooth convex function with minimum $h^*$, compact set of minimizers $X$ and with $L_h$-Lipschitz gradient. If given $a > h^*$ the sublevel set $\OM = \{x \ | \ h(x) \leq a\}$ is compact, and $\T= 1/2$ then 	
	\begin{equation} \label{ballthesis}
	\bar{\nu}(x, g) \geq \frac{\n{\nabla h(x)}}{\n{g}L_h}
	\end{equation} 
	for every $(x, g) \in \tom$.
\end{proposition} 
\begin{proof}
	Fix $\x \in \partial \OM$ so that $h(\x) = a$, and let $\bh = \n{\nabla h(\x)}/L$.
	By the standard descent Lemma for any $y \in \R^n$
	\begin{equation*}
	h(y) \leq h(\x) + \s{\nabla h(\x)}{y - \x} + \frac{L}{2}\n{y - x}^2 \, ,
	\end{equation*}
	and in particular $\bar{B}(\x) = \tx{cl}(B_{\x - \nabla h(\x)/L}(\bh)) \subset{\OM}$. Let now $P_{\bar{B}(\x)}$ be the orthographic retraction on $\bar{B}(\x)$, and for $g$ such that $(\x, g) \in \tom$, let $\bg = \n{g}$.
	
	In order to prove \eqref{ballthesis} it suffices to prove
	\begin{equation} \label{interm}
	\frac{\s{g}{\vg}}{\pi_{\x}(g) \n{\vg}} \geq \frac{1}{2} \, .
	\end{equation}
	for $\vg = P(\x, \frac{\bh}{\bg} \ro(\x)) - \x = P(\x, \bh \hat{g}_{\OM}(\x)) - \x$. 
	First, we have
	\begin{equation} \label{lhs}
	\frac{\s{g}{\vg}}{\pi_{\x}(g) \n{\vg}} \geq \frac{\s{g}{\bvg}}{\pi_{\x}(g) \n{\bvg}}
	\end{equation}
	for $\bvg = P_{\bar{B}(\x)}(\x, \bh \hat{g}_{\OM}(\x)) - \x $, given that $\bar{B}(\x) \subset \OM$. The RHS of \eqref{lhs} can be bounded with a tedious but straightfoward computation. Let $\bgo = \n{\hat{g}_{\OM}(\x)}$ and $\bj = \n{\hat{g}_J(\x)}$, so that $\pi_x(g) = \bgo \bg$, $\bgo^2 + \bj^2 = 1$, $\bgo > 0$. Of course we can explicitly compute the orthographic projection on $\bar{B}(\x)$ and in particular we have 
	$\bvg = \bh \hat{g}_{\OM}(\x) - \gamma(\bh \bgo ) J(\x)$, for $\gamma(r) = \bh - \sqrt{\bh^2 - r^2}$. Then 
	$\s{g}{\bvg} = \bgo \bg \bh  - \bj \bg \gamma(\bh \bgo)  $, $\pi_{\x}(g) = \bg \bgo$, and  $\n{\bvg} = \sqrt{\bh^2\bgo^2 + \gamma(\bh \bgo)^2}$. Therefore 
	\begin{equation*}
	\frac{\s{g}{\bvg}}{\pi_{\x}(g) \n{\bvg}} =  \frac{\bgo \bh - \bj \gamma(\bgo \bh)}{\bgo \sqrt{\bh^2 \bgo^2 + \gamma(\bh \bgo)^2}} \, ,
	\end{equation*} 
	and we can finally write
	\begin{equation*}
	\inf_{g: (\x, g) \in \tom}\frac{\s{g}{\bvg}}{\pi_{\x}(g) \n{\bvg}} =  
	\inf \left\{ \frac{\bgo \bh -  \bj \gamma(\bgo \bh)}{\bgo \sqrt{\bh^2 \bgo^2 + \gamma(\bh \bgo)^2} } \ | \ \bgo > 0, \ \bj \geq 0, \ \bj^2 + \bgo^2 = 1   \right\} = \frac{1}{2} \, ,
	\end{equation*}
	which combined with \eqref{lhs} proves \eqref{interm}. 
\end{proof}
\begin{remark} \label{proximally_smooth_rem}
	The above proposition can be generalized for any function $h$ with a convex increasing smoothness modulus $\bar{\eta}$ (see e.g. \cite{aze95} for some related properties) such that
	\begin{equation*}
	h(y) \leq h(x) + \s{\nabla h(x)}{y - x} + \bar{\eta}(\n{x - y})
	\end{equation*}
	for every $x, y \in \R^n$. In this case $\bn(x, \hat{g})$ can be lower bounded by a function of $\bar{\eta}/\n{\nabla h(x)}$ and $\n{\hat{g}_{\OM}(x)}$, decreasing in both arguments. In particular when $h$ is $q-$uniformly smooth with $\bar{\eta}(t) \approx t^q$ for some $q \in (1, 2]$ we have
	\begin{equation*}
	\bn(x, \hat{g}) \gtrsim  \n{\hat{g}_{\OM}(x)}^{\frac{2 - q}{q - 1}} 
	\end{equation*}
	for $\hat{g}_{\OM}(x) \rightarrow 0$ and uniformly in $x$ using that by compactness $\n{\nabla h(x)}$ is bounded away from $0$ in $\partial \OM$.
\end{remark}
	\subsection{Fr\'echet subdifferential and KL property} \label{KLproperty} 
	In this section we report the definition of the Kurdyka-\L ojiasiewicz property as it is presented in \cite{attouch2010proximal}, \cite{attouch2013convergence}. We first need to recall a few additional definitions from variational analyisis, the standard references being  \cite{clarke1995proximal}, \cite{mordukhovich2006variational} and \cite{rockafellar2009variational}.
	\begin{definition}
		Given a lower semicontinuous function $\f:\R^n \rightarrow \R \cup \{+\infty\}$ with $\tx{dom}(\f) \neq \emptyset$ the vector $z$ is said to be in the Fr\'echet subdifferential of $\f$ at $x$, written $\hat{\partial}\f(x)$, iff 
		\begin{equation*}
			\f(y) \geq \f(x) + \s{z}{y-x} + o(\n{y-x}) \quad \tx{for all } y \in \R^n \, .
		\end{equation*}
		The limiting subdifferential of $\f$ at $x \in \tx{dom}(\f)$, written $\partial \f(x)$, is the set
		\begin{equation*}
			\partial \f(x) : = \{z \in \R^n : \exists \ x_n \rightarrow x, \ \f(x_n) \rightarrow \f(x), \ z_n \in \hat{\partial} \f(x_n) \rightarrow z \} \, .
		\end{equation*}
	\end{definition}
	It is clear from the definition that the limiting subdifferential is always closed.  
	
	We can now define the KL property. 
	\begin{definition}
		A function $\f$ as above is said to have the KL property at $x^* \in \tx{dom}(\partial \f)$ if there exists $\eta \in (0,+\infty)$, a neighborhood $U$ of $x^*$ and a continuous concave function $\varphi:[0, \eta) \rightarrow \R_{\geq 0}$ (called \textit{desingularizing function}) such that:
		\begin{itemize}
			\item[(i)] $\varphi(0) = 0$,
			\item[(ii)] $\varphi$ is $C^1$ on $(0,\eta)$,
			\item[(iii)] for all $t \in (0,\eta), \varphi'(t) > 0$, 
			\item[(iv)] for all $x \in U\cap [\f(x^*) <\f< \f(x^*) + \eta]$, the KL inequality holds
			\begin{equation} \label{klineq}
				\varphi'(\f(x) - \f(x^*))\dist(\partial \f(x), 0) \geq 1 \, .
			\end{equation}
		\end{itemize}
	\end{definition}
	
	When $\tilde{f} = f + i_{\OM}$ with $f \in C^1(\OM)$, $\OM$ convex and closed, the KL inequality \eqref{klineq} can be rewritten as
	\begin{equation} \label{convexsmoothKL}
		\varphi'(f(x) - f(x^*))\n{\pi(T_{\OM}(x), -\nabla f(x))}  \geq 1 \, .
	\end{equation}
	Indeed, by the assumptions on $\OM$ (closedness and convexity), we can write 
	\begin{equation*}
		\partial i_{\OM}(x) = N_{\OM}(x) \ \forall x \in \OM \, ,
	\end{equation*}
	and we have the sum rule (by e.g. \cite[Proposition 2.2]{mordukhovich2006frechet}) 
	\begin{equation*}
		\partial \f_{\OM}(x) = \partial i_{\OM}(x) + \nabla f(x) = N_{\OM}(x) + \nabla f(x) \, .
	\end{equation*}
	Finally,  
	\begin{equation} \label{eqdpi}
		\dist (\partial \f_{\OM}(x), 0) = \dist(N_{\OM}(x) + \nabla f(x), 0) = \dist(N_{\OM}(x), -\nabla f(x)) = \n{\pi(T_{\OM}(x),-\nabla f(x))} \, ,
	\end{equation}
	where we used Proposition \ref{NWessential} in the last equality and the equivalence is proved. 
	
	The study of bounds on the desingularizing function $\varphi$ for smooth functions constrained to convex sets is still very much an open area of research. In \cite{li2018calculus} it was proved that a KL property with desingularizing function of the form $\varphi(t) = Mt^{1/2}$ is implied by the Luo-Tseng error bound. For other examples see Section \ref{s:QP}.
	\begin{remark} \label{LPLrem}
		Consider the LPL condition (see \cite{balashov2019gradient}, Definition 1 for the details)
		\begin{equation*}
			\n{\pi(T_x\OM, -\nabla f(x))}^\alpha \geq \mu (f(x) - f^*) \, ,
		\end{equation*}
		with $f^* = \min_{x \in \OM} f(x)$, $T_x\OM$ the tangent plane to $\OM$ in $x$ equal to $T_{\OM}(x)$, $f$ smooth and $\OM$ a $C^1$ surface. For $\alpha > 1$ this condition is a particular case of the KL condition. Indeed the KL inequality can still be written in the form \eqref{convexsmoothKL} when $\OM$ is a surface, with $T_{\OM}(x) = T_x\OM$, as it can be proved applying the same reasoning we have seen for convex sets. Then if $\alpha > 1$ we retrieve the LPL condition considering $\varphi(t) = \mu^{-\frac{1}{\alpha}} \frac{\alpha}{\alpha - 1} t^{1-1/\alpha} $ and any $x \in \argmin_{x \in \OM} f(x)$ in \eqref{convexsmoothKL}.
	\end{remark}

	\subsection{Proofs for Section \ref{s:CP}} \label{sigmaalpha}
	
	In this section we prove Proposition \ref{keytec} and Corollary \ref{holderian}. We start by recalling Karamata's inequality for concave functions. Given $A, B \in \R^N$ it is said that $A$ majorizes $B$, written $A \succ B$, if
	\begin{equation*}
		\begin{aligned}
			\sum_{i= 1}^j A_i & \geq \sum_{i= 1}^j B_i \ \tx{for } j \in [1:N] \, , \\
			\sum_{i = 1}^N A_i & = \sum_{i= 1}^N B_i \, .
		\end{aligned}		
	\end{equation*}
	If $h$ is concave and $A \succ B$ by Karamata's inequality
	\begin{equation*} 
		\sum_{i= 1}^N h(A_i) \leq \sum_{i = 1}^N h(B_i) \, .
	\end{equation*}
	We now prove a technical lemma. It gives an upper bound on a sum of square roots which is the key to upper bound the length of the tail of the sequence under the descent conditions \eqref{eq:H1}, \eqref{d'2}. 
	\begin{lemma} \label{finitelenlem}
		Let $f_0 \in [0, \eta]$ and assume 
		\begin{equation*}
			I = \int_{0}^{\eta} \alpha(t) dx < + \infty \, .
		\end{equation*}
		Then 
		\begin{equation} \label{sabound}
			\sum_{s=0}^{\infty} \sqrt{\sigma^{(s)}_{\alpha}(f_0) - \sigma^{(s+1)}_{\alpha}(f_0)} \leq I + 2\sqrt{f_0 - \sigma_{\alpha}(f_0)} \, .
		\end{equation}
	\end{lemma}
	\begin{proof}
		Let $S = \inf\{s \in \mathbb{N}_{0} \ | \ \sa^{(s)}(f_0) = 0\}$. If $S \in \{0, 1\}$ then \eqref{sabound} is trivially satisfied. Otherwise for every $0 \leq s \leq S - 1$ we have 
		\begin{equation} \label{seq}
			\sqrt{\sigma^{(s)}_{\alpha}(f_0) - \sigma^{(s+1)}_{\alpha}(f_0)} \leq \frac{1}{\alpha(\sa^{(s+1)}(f_0))}
		\end{equation}
		thanks to \eqref{0sa} and \eqref{s2eq}, with equality for $0 \leq s \leq S - 2$. \\ 
		Let $2 \leq P \leq S$. Then by \eqref{seq} for $0 \leq s \leq S - 2$
		\begin{equation} \label{sp2}
			\sum_{s = 0}^{P-2} \sqrt{\sa^{(s)}(f_0) - \sa^{(s+1)}(f_0)} = 	\sum_{s = 0}^{P-2} \frac{1}{\alpha(\sa^{(s+1)}(f_0))} \, .
		\end{equation}
		From now on for simplicity we write $\bara^{(s)}(f_0)$ to denote $\alpha(\sa^{(s)}(f_0))$. \\ 
		We have 
		\begin{equation} \label{Iineq}
			I =  \int_{0}^{\eta} \alpha(t) dx \geq	\sum_{s = 0}^{P-2} \bara^{(s)}(f_0) (\sa^{(s)}(f_0) - \sa^{(s + 1)}(f_0) ) = \sum_{s= 0}^{P-2} \frac{\bara^{(s)}(f_0)}{\bara^{(s+1)}(f_0)^2} \, ,
		\end{equation}
		where in the first inequality the RHS is a Riemann sum and we used that $\alpha$ is decreasing, and in the second equality we used again \eqref{s2eq}. \\
		We now have 
		\begin{equation} \label{eq1pieceS}
		\begin{aligned}
					& \sum_{s = 0}^{P-2}	\frac{1}{\bara^{(s+1)}(f_0)} - \sum_{s = 0}^{P-2} \frac{\bara^{(s)}(f_0)}{\bara^{(s+1)}(f_0)^2} \\
				= & \frac{1}{\bara^{(P-1)}(f_0)} - \frac{\alpha(f_0)}{\bara(f_0)^2} + \sum_{s= 1}^{P-2}  \bara^{(s)}(f_0) \left( \frac{1}{\bara^{(s)}(f_0)^2} - \frac{1}{\bara^{(s+1)}(f_0)^2} \right) \, .
		\end{aligned}
		\end{equation}
		The sum on the RHS turns out to be easily lower bounded by a telescopic sum  
		\begin{equation} \label{eq2pieceS}
			\begin{aligned}
				&\sum_{s = 1}^{P-2}  \bara^{(s)}(f_0) \left( \frac{1}{\bara^{(s)}(f_0)^2} - \frac{1}{\bara^{(s+1)}(f_0)^2} \right) \\
				= &  \sum_{s = 1}^{P-2}  \bara^{(s)}(f_0) \left( \frac{1}{\bara^{(s)}(f_0)} - \frac{1}{\bara^{(s+1)}(f_0)}\right) \left(\frac{1}{\bara^{(s)}(f_0)} + \frac{1}{\bara^{(s+1)}(f_0)} \right) \\
				\leq & 2 \sum_{s = 1}^{P-2}  \left(\frac{1}{\bara^{(s)}(f_0)} - \frac{1}{\bara^{(s+1)}(f_0)}\right) =  \frac{2}{\bara(f_0)} - \frac{2}{\bara^{(P-1)}(f_0)} \, ,
			\end{aligned}	
		\end{equation}
		where in the inequality we used 
		$$\bara^{(s)}(f_0) \left(\frac{1}{\bara^{(s)}(f_0)} + \frac{1}{\bara^{(s+1)}(f_0)} \right) \leq 2 \, .$$
		We therefore have
		\begin{equation} \label{sumineq}
			\begin{aligned}
				& \sum_{s = 0}^{P-2}	\frac{1}{\bara^{(s+1)}(f_0)} - I \leq  \sum_{s = 0}^{P-2}	\frac{1}{\bara^{(s+1)}(f_0)} - \sum_{s = 0}^{P-2} \frac{\bara^{(s)}(f_0)}{\bara^{(s+1)}(f_0)^2}  \\ 
				\leq & \frac{1}{\bara^{(P-1)}(f_0)} - \frac{\alpha(f_0)}{\bara(f_0)^2} +  \frac{2}{\bara(f_0)} - \frac{2}{\bara^{(P-1)}(f_0)}   \\ 
				\leq & \frac{2}{\bara(f_0)} - \frac{1}{\bara^{(P-1)}(f_0)} \, ,
			\end{aligned}
		\end{equation}
		where we used \eqref{Iineq} in the first inequality, \eqref{eq1pieceS} together with \eqref{eq2pieceS} in the second. 
		Then if $S = +\infty$ taking the limit for $P \rightarrow +\infty$ in \eqref{sumineq} we obtain
		\begin{equation} \label{infiniteS}
			\sum_{s = 0}^{\infty} \frac{1}{\bara^{(s+1)}(f_0)} - I \leq  \frac{2}{\bara(f_0)} \, ,
		\end{equation}
		and the thesis follows from \eqref{sp2} with $P = + \infty$. \\
		If $S < \infty$ then for $P = S$ we obtain 
		\begin{equation} \label{finiteS}
			\sum_{s = 0}^{S-1}	\frac{1}{\bara^{(s+1)}(f_0)} - I = \frac{1}{\bara^{(S)}(f_0)} + \sum_{s = 0}^{P-2}	\frac{1}{\bara^{(s+1)}(f_0)} - I \leq \frac{2}{\bara(f_0)} + 	\frac{1}{\bara^{(S)}(f_0)}  - \frac{1}{\bara^{(S-1)}(f_0)} \leq \frac{2}{\bara(f_0)} \, ,
		\end{equation}
		where we used \eqref{sumineq} in the first inequality and that $1/\bara^{(s)}(f_0)$ is decreasing in $s$ in the last inequality. We can now write 
		\begin{equation} \label{eq:lastineq}
			\sum_{s=0}^{\infty} \sqrt{\sigma^{(s)}_{\alpha}(f_0) - \sigma^{(s+1)}_{\alpha}(f_0)} = \sum_{s=0}^{S-1} \sqrt{\sigma^{(s)}_{\alpha}(f_0) - \sigma^{(s+1)}_{\alpha}(f_0)} \leq \sum_{s = 0}^{S-1}	\frac{1}{\bara^{(s+1)}(f_0)}  \leq I + \frac{2}{\bara(f_0)} \, ,
		\end{equation}
		where we used $\sigma^{(s)} = 0$ for $s \geq S$ in the equality, \eqref{seq} in the first inequality, \eqref{finiteS} and \eqref{infiniteS} in the second for the cases $S < + \infty$ and $S = + \infty$ respectively. By the assumption $S \geq 2$ we have 
		$$2/\bar{\alpha}(f_0) = 2\sqrt{f_0 - \sigma_{\alpha}(f_0)}$$
		by \eqref{seq} and applying this inequality to the RHS of \eqref{eq:lastineq} the thesis is proved. 
	\end{proof}

	To prove Proposition \ref{keytec} we first need another technical lemma where we use Karamata's inequality in combination with condition \eqref{eq:H1} to upper bound the length of the tails of $\{x_k\}$.
	\begin{lemma} \label{tlem}
		Let $\tilde{f}$ be as in Proposition \ref{keytec}, $\{x_k\}_{k \in I}$ be a sequence indexed by $I = [0, m] \cap \mathbb{N}_0$ for some $m \in \mathbb{N} \cup \{\infty \}$. 
		Assume that $\{x_k\}_{k \in I}$ satisfies conditions \eqref{eq:H1} and \eqref{eq:gd'2} for every $k \in I$ and that $\f$ satisfies the KL inequality with respect to $x^*$ in $\{\tilde{x}_k\}_{k \in I\sm \{m\}} \cap [\f(x^*) <\f< \eta]$ with $\f(x_0) < \eta$. Let $\alpha(t) = \frac{b}{\sqrt{a}} \varphi'(t)$. Then  
		\begin{equation} \label{convergencerates}
			\begin{aligned}
				\f(x_k) -\f(x^*) \leq & \sigma^{(k)}_{\alpha}(\f(x_0) - \f(x^*))  \\ 
				\sum_{i= k}^{m-1} \n{x_i - x_{i+1}} + \sqrt{\frac{\f(x_m) - \f(x^*)}{a}} \leq & \frac{b}{a}\varphi(\f(x_k) - \f(x^*)) + 2 \sqrt{\frac{\f(x_k) - \f(x^*) -  \sigma_{\alpha}(\f(x_k) - \f(x^*))}{a}} \, ,
			\end{aligned}
		\end{equation}	
		for all $k \leq m-1$	where we set $\f(x_m) = \f(x^*)$ for $m = \infty$. 
	\end{lemma}
	
	\begin{proof}
		First observe that $\{\f(x_k)\}$ is decreasing by condition \eqref{eq:H1}, so that for every $k \in I$ we have $\f(\tilde{x}_k) \leq \f(x_k) < \eta$. Therefore for every $k \in I$ if $\f(\tilde{x}_k) > \f(x^*)$ then the KL inequality holds in $\tilde{x}_k$. \\
		Let $f_i = \f(x_i) - \f(x^*) \geq 0$ for $i \in I$, $f_i = 0$ for $i > m$. If $f_{i+1} = 0$ then trivially 
		\begin{equation*} 
			f_{i+1} = 0\leq \sa(f_i) \, .
		\end{equation*}
		If $f_{i+1} > 0$ then
		\begin{equation} \label{n2alphaineq}
			\begin{aligned}
				&	\f(x_i) - \f(x_{i+1}) \geq \frac{a}{b^2}  \dist(\partial \f (\tilde{x}_{i+1}), 0 )^2  \\ 
				\geq & \frac{a}{b^2 \varphi'(\f(\tilde{x}_i) - \f(x^*))^2} \geq \frac{1}{\alpha(f_{i+1})^2} \, ,
			\end{aligned}		
		\end{equation}
		where the first inequality is \eqref{eq:gd'2}, the second the KL property which holds in $\tilde{x}_i$ by \eqref{eq:gd''2} since $\f(\tilde{x}_i) \geq \f(x_{i + 1}) > \f(x^*)$, and the third follows from the monotonicity of $\alpha$. We can rewrite \eqref{n2alphaineq} as 
		\begin{equation*}
			f_{i + 1} + \frac{1}{\alpha(f_{i + 1})^2} \leq f_i \, ,
		\end{equation*} 
		and from the definition of $\sigma_{\alpha}$ also in this case $f_{i+1} \leq \sa(f_i)$. Now by induction we obtain
		\begin{equation} \label{n1decreasefi}
			f_i\leq \sigma_{\alpha}^{(i - l)}(f_l)
		\end{equation}
		for every $0 \leq l \leq i$, and to prove the first part of \eqref{convergencerates} it suffices to take $k = i$ and $l = 0$. \\
		By the monotonicity of $\alpha$ equation \eqref{0sa} together with \eqref{s2eq} imply that the sequence $\sigma_{\alpha}^{(i - l)}(f_l) - \sigma_{\alpha}^{(i - l + 1)}(f_l)$ is decreasing in $i$.
		Let now $i \geq k + 1$, and let
		\begin{equation*}
			n(i, k) = \min \{j \ | \ f_{j+1} \leq \sa^{(i - k)}(f_k)  \} \, .
		\end{equation*}
		Then by \eqref{n1decreasefi} with $k = l$ we have $n(i, k) \leq i  - 1$. Consider the vectors 
		$v_{ik}, w^*_{ik}$ in $\R^{i - k}_{\geq 0}$ defined by 
		\begin{equation*}
			\begin{aligned}
				v_{ik} &= (f_k - \sa(f_k), ..., \sa^{(i - k - 1)}(f_k) - \sa^{(i - k)}(f_k)) \, , \\
				w^*_{ik} &= (f_k - f_{k+1}, ..., f_{n(i, k)} -\sa^{(i - k)}(f_k), 0, ..., 0) \, .
			\end{aligned}
		\end{equation*}
		so that in particular $v_{ik}$ is in decreasing order. \\
		For $n(i, k) - k + 1 \leq u \leq i - k$ we have
		\begin{equation} \label{n3sum}
			\sum_{j=1}^{u} v_{ik}(j) = f_k - \sa^{(u)}(f_k) \leq f_k - \sa^{(i - k)}(f_k) = \sum_{j=1}^{u} w^*_{ik}(j) \, ,
		\end{equation}
		where the inequality follows by monotonicity of $\sa^{(u)}$ in $u$ and we have equality if $u = i - k$. For $1 \leq u \leq n(i, k) - k$ we have
		\begin{equation} \label{n3Ksum}
			\sum_{j=1}^{u} v_{ik}(j) = f_k - \sa^{(u)}(f_k)  \leq f_k - f_{u + k} = \sum_{j=1}^{u} w_{ik}^*(j) \, .
		\end{equation}
		where the inequality follows by \eqref{n1decreasefi} with $l = k$ and $u = i - l$.
		Combining \eqref{n3sum} and \eqref{n3Ksum} we obtain that if $w_{ik}$ is the permutation in decreasing order of $w_{ik}^*$ then $w_{ik} \succ v_{ik}$ so that by Karamata's inequality 
		\begin{equation*}
			\sqrt{f_{n(i, k)} - \sa^{(i - k)}(f_k)} + \sum_{j=k + 1}^{n(i, k)} \sqrt{f_{j-1} - f_j} = \sum_{j=1}^{i - k} \sqrt{w_{ik}^*(j)} \leq \sum_{j=1}^{i - k} \sqrt{v_{ik}(j)} = \sum_{j=1}^{i - k} \sqrt{\sa^{(j - 1)}(f_k) - \sa^{(j)}(f_k)} \, .
		\end{equation*}
		Sending $i$ to infinity since $f_{n(i, k)} -\sa^{(i - k)}(f_k) \rightarrow 0$ (because both terms converge to 0) we obtain 
		\begin{equation} \label{sqrsa}
			\sum_{j= k}^{\infty} \sqrt{f_{j} - f_{j+1}} \leq \sum_{j=0}^{\infty} \sqrt{\sa^{(j)}(f_k) - \sa^{(j+1)}(f_k)} \, .
		\end{equation}
		We now apply Lemma \ref{finitelenlem} to the RHS: 
		\begin{equation} \label{saif0}
			\begin{aligned}
				\sum_{j=0}^{\infty} \sqrt{\sa^{(j)}(f_k) - \sa^{(j+1)}(f_k)} & \leq \int_{0}^{f_k}\alpha(t) dt + 2 \sqrt{\f(x_k) - \f(x^*) -  \sigma_{\alpha}(\f(x_k) - \f(x^*))} \\ & = \frac{b}{\sqrt{a}}\varphi(\f(x_k) - \f(x^*)) + 2 \sqrt{\f(x_k) - \f(x^*) -  \sigma_{\alpha}(\f(x_k) - \f(x^*))} \, ,
			\end{aligned} 	
		\end{equation}
		where we used $\alpha(t) = \frac{b}{\sqrt{a}} \varphi'(t)$ in the equality. \\
		We can now prove the second part of \eqref{convergencerates}:  
		\begin{equation*}
			\begin{aligned}
				\sum_{i=k}^{m - 1} \n{x_i - x_{i+1}} + \sqrt{\frac{\f(x_m) - \f(x^*)}{a}} & \leq \frac{1}{\sqrt{a}} \sum_{i=k}^{m - 1} \sqrt{f_i - f_{i+1}} + \sqrt{\frac{\f(x_m) - \f(x^*)}{a}} =  \frac{1}{\sqrt{a}}\sum_{i=k}^{\infty} \sqrt{f_i - f_{i+1}}  \\
				& \leq \frac{b}{a}\varphi(\f(x_k) - \f(x^*)) + 2 \sqrt{\frac{\f(x_k) - \f(x^*) -  \sigma_{\alpha}(\f(x_k) - \f(x^*))}{a}} \, ,
			\end{aligned} 	
		\end{equation*}
		where in the fist inequality we applied \eqref{eq:H1}, the equality follows from $f_i = \f(x_i) - \f(x^*) = 0$ for $i > m$, and the second inequality follows from \eqref{sqrsa} together with \eqref{saif0}.  	
	\end{proof}
	
	We can now prove the main result. 
	\begin{proof}[Proposition \ref{keytec}]
		We show by induction $\{x_k\}, \{\tilde{x}_k\} \subset B_{\delta}(x^*)$ with $f(x_k) \geq f(x^*)$ for every $k$.  By \eqref{corcond} we have $\n{x_0 - x^*} < \delta $. Now if we know $\{x_k\}_{k \leq u} \subset B_{\delta}(x^*)$, $\min_{k \leq u} f(x_k) \geq f(x^*) $ and $\{\tilde{x}_k\}_{k \leq u-1}\subset B_{\delta}(x^*)$ for some $u \in \mathbb{N}_0$ then by condition \eqref{cond:boundc} we have $f(x_{u + 1}) \geq f(x^*)$. We also have
		\begin{equation} \label{xmm1}
			\begin{aligned}
				&	\n{x_{u+1} - x^*} \leq  \n{x_{u+1} - x_u} +  \n{x_u - x^*}  \\ \leq
				&  \sqrt{\frac{f(x_u) - f(x_{u + 1})}{a}} + \n{x_u - x^*} \\ \leq
				& \sqrt{\frac{f(x_u) - f(x_{u + 1})}{a}} + \sum_{k= 0}^{u-1} \n{x_{k+1} - x_k} + \n{x_0 - x^*} 	\\ \leq 
				& \sqrt{\frac{f(x_u) - f(x^*)}{a}} + \sum_{k= 0}^{u-1} \n{x_{k+1} - x_k} + \n{x_0 - x^*}  \\ 
				\leq &  \frac{b}{a}\varphi(f(x_0) - f(x^*)) + 2 \sqrt{\frac{f(x_0) - f(x^*) -  \sigma_{\alpha}(f(x_0) - f(x^*))}{a}} + \n{x_0 - x^*} < \delta \, ,
			\end{aligned}
		\end{equation}
		where we used \eqref{eq:H1} in the first inequality, condition \eqref{cond:boundc} in the third inequality, Lemma \ref{tlem} in the fourth one and assumption \eqref{corcond} in the last one. Finally,
		\begin{equation*}
			\begin{aligned}
				\n{\tilde{x}_u - x^*} \leq & \n{\tilde{x}_u - x_u} + \n{x_u - x^*}  
				\leq \sqrt{\frac{f(x_u) - f(\tilde{x}_u)}{a}}  +  \n{x_u - x^*} \\
				\leq &  \sqrt{\frac{f(x_u) - f(x_{u + 1})}{a}} +  \n{x_u - x^*} <  \delta \, ,
			\end{aligned}
		\end{equation*}
		where in the second and in the third inequality we used \eqref{eq:gd''2}, while the last inequality follows from \eqref{xmm1}. This completes the induction. Applying Lemma \ref{tlem} with $m = + \infty$ we obtain \eqref{convergencerates2}, the second inequality of \eqref{convergencerates3} and in particular 
		$ f(x_k) \rightarrow f(x^*)$ by \eqref{eq:tsigmalim} together with the condition $f(x_k) \geq f(x^*)$ for every $k \in \mathbb{N}_0$. Finally, given that $\{x_k\}$ has finite length it has a limit $\tilde{x}^*$ with $\f(\tilde{x}^*) \leq \f(x^*)$ by lower semicontinuity, and by the triangular inequality we have that the first part of \eqref{convergencerates3} is also satisfied.  
	\end{proof}

	The proof of Corollary \ref{holderian} given Proposition \ref{keytec} is based on a straightforward computation of the asymptotic behaviour of the worst case sequence for $\alpha(t) = U t^{\epsilon}$, with $U > 0$ and $\epsilon \in ( - 1, -\frac{1}{2}]$. For the sake of completeness we include here a detailed version of the proof. 
	\begin{proof}[Corollary \ref{holderian}]
		By Proposition \ref{keytec} we have  
		\begin{equation} \label{wellknown}
			\f(x_k) -\f(x^*) \leq \sigma^{(k)}_{\alpha}(\f(x_0) - \f(x^*)) \, ,
		\end{equation}
		with $\alpha(t) = Mb\sqrt{\frac{1}{a}} t^{\theta - 1} $. Now if $\theta = \frac{1}{2}$ then equation \eqref{s2eq} becomes  
		\begin{equation*}
			t - \sigma_{\alpha}(t) = \frac{a}{b^2M^2} \sigma_{\alpha}(t) \, , 
		\end{equation*}
		which has always the solution 
		\begin{equation*}
			\sigma_{\alpha}(t) = t\left(1 + \frac{a}{b^2M^2} \right)^{-1} \, .
		\end{equation*}
		The bound \eqref{thetafconv} for $\theta= \frac{1}{2}$ then follows by induction. \\
		In this case we also have 
		\begin{equation} \label{baphi}
			\frac{b}{a}\varphi(\f(x_k) - \f(x^*)) = 	O\left( \left(1 + \frac{a}{b^2M^2} \right)^{-\frac{k}{2}}\right)
		\end{equation}
		and 
		\begin{equation} \label{sqrtfk}
			2 \sqrt{\frac{\f(x_k) - \f(x^*) -  \sigma_{\alpha}(\f(x_k) - \f(x^*))}{a}} = O(\sqrt{\f(x_k) - \f(x^*)}) = O\left( \left(1 + \frac{a}{b^2M^2} \right)^{-\frac{k}{2}} \right) \, .
		\end{equation}
		We can now prove \eqref{thetatailconv} for $\theta = \frac{1}{2}$:
		\begin{equation*}
			\begin{aligned}
				\sum_{i = k}^\infty \n{x_{i} - x_{i+1}} &\leq \frac{b}{a}\varphi(\f(x_k) - \f(x^*)) + 2 \sqrt{\frac{\f(x_k) - \f(x^*) -  \sigma_{\alpha}(\f(x_k) - \f(x^*))}{a}}\\ 
				&= 	O\left( \left(1 + \frac{a}{b^2M^2}\right)^{-\frac{k}{2}} \right) \, ,
			\end{aligned}	
		\end{equation*}
		where we used Corollary \ref{cor:corA2conv} in the first inequality, and we summed \eqref{sqrtfk}, \eqref{baphi} in the second.  \\
		Let $0 < \theta < \frac{1}{2}$, $g = \frac{1}{1-2\theta}$, $f_0 = \f(x_0) - \f(x^*)$ and let
		\begin{equation} \label{Pineq}
			P = \max \left(f_0, \left(\frac{2^{r+2}rb^2M^2}{a} \right)^r \right) \, .
		\end{equation}
		We now prove by induction
		\begin{equation} \label{fineq}
			\sigma_{\alpha}^{(k)}(f_0) \leq \frac{P}{(k+1)^r} \, ,
		\end{equation}
		so that the bound \eqref{thetafconv} follows by
		\begin{equation*}
			\f(x_k) -\f(x^*) \leq \sigma_{\alpha}^{(k)}(f_0) \leq \frac{P}{(k+1)^r} \, ,
		\end{equation*}
		where we used \eqref{wellknown} in the first inequality and \eqref{fineq} in the second. \\
		For $k=0$ \eqref{fineq} is true by definition of $P$. Assume now  	$\sigma_{\alpha}^{(k)}(f_0) =  \frac{P}{J^r}$ for some $J \geq k+1$. Then on the one hand by \eqref{s2eq}
		\begin{equation} \label{induc1}
			0 = \sigma_{\alpha}^{(k)}(f_0) - \sigma_{\alpha}^{(k + 1)}(f_0) - \frac{1}{\alpha(\sigma^{(k+1)}_{\alpha}(f_0))^2} = 
			\frac{P}{J^r} - \sigma_{\alpha}\left(\frac{P}{J^r} \right) - \frac{1}{\alpha(\sigma_{\alpha}(\frac{P}{J^r}))^2} \, ,
		\end{equation}
		while on the other hand
		\begin{equation} \label{induc2}
			\frac{P}{J^r} - \frac{P}{(J+1)^r} - \frac{1}{\alpha(\frac{P}{(J+1)^r})^2} = \frac{P}{J^r} - \frac{P}{(J+1)^r} - \frac{aP^{2-2\theta}}{b^2M^2(J+1)^{r+1}} < \frac{rP}{J^{r+1}} - \frac{aP^{2-2\theta}}{b^2M^2(J+1)^{r+1}} \leq 0 \, ,
		\end{equation}
		where the last inequality follows from the definition \eqref{Pineq} of $P$. Now the function 
		\begin{equation*}
			q(s) = \frac{P}{J^r} - s  - \frac{1}{\alpha(s)^2}
		\end{equation*}
		is decreasing in $s$ and $q(\frac{P}{(J+1)^r}) < q(\sigma_{\alpha}(\frac{P}{J^r})) $ by \eqref{induc1} and \eqref{induc2}. It follows 
		\begin{equation*}
			\sigma_{\alpha}^{(k+1)}(f_0) =\sigma_{\alpha}(\sigma_{\alpha}^{(k)}(f_0)) = \sigma_{\alpha}(\frac{P}{J^r}) < \frac{P}{(J+1)^r} \leq  \frac{P}{(k+2)^r} \, ,
		\end{equation*}
		where we used $J \geq k+1$ in the last inequality and the induction is complete. \\
		Having established $\f(x_k) - \f(x^*) = O(1/k^{\frac{1}{1 - 2\theta}})$, we have
		\begin{equation*}
			\frac{b}{a}\varphi(\f(x_k) - \f(x^*)) = 	O(1/k^{\frac{\theta}{1 - 2\theta}}) \, .
		\end{equation*}
		Whence
		\begin{equation*} 
			\begin{aligned}
				& 2 \sqrt{\frac{\f(x_k) - \f(x^*) -  \sigma_{\alpha}(\f(x_k) - \f(x^*))}{a}} = O\left( \frac{1}{\alpha(\sigma_{\alpha}(\f(x_k) - \f(x^*)))} \right) \\
				= &O\left(\frac{1}{\alpha(\f(x_k) - \f(x^*))}\right) = O(1/k^{\frac{1-\theta}{1-2\theta}} ) = 	O(1/k^{\frac{\theta}{1 - 2\theta}}) \, ,
			\end{aligned}
		\end{equation*}
		where we used \eqref{s2eq} in the first equality, $\sigma_{\alpha}(\f(x_k) -  \f(x^*)) < \f(x_k) - \f(x^*)$ in the second and $1-\theta > \theta$ in the last one. \\ 
		Reasoning exactly as for $\theta = \frac{1}{2}$ we obtain the bound \eqref{thetatailconv} for $\theta \in(0, \frac{1}{2})$. \\ 
		Finally, in the above proof the dependence from $f_0$ of the implicit constants is always monotone increasing, so that it can be eliminated by replacing $f_0 = \tilde{f}(x_0) - \f(x^*)$ with $\eta \geq f_0$ in the bounds \eqref{wellknown} and \eqref{fineq}. 
		
	\end{proof}
	
	\clearpage
	\bibliographystyle{plain}  
	\bibliography{tesibib} 
\end{document}